\documentclass[10pt, a4paper]{article} % use larger type; default would be 10pt

\usepackage[utf8]{inputenc} % set input encoding (not needed with XeLaTeX)

\usepackage{graphicx} % support the \includegraphics command and options

\usepackage[english]{babel}

\usepackage{booktabs} % for much better looking tables
\usepackage{array} % for better arrays (eg matrices) in maths
\newcolumntype{L}{>{\centering\arraybackslash}m{5cm}}
\usepackage{paralist} % very flexible & customisable lists (eg. enumerate/itemize, etc.)
\usepackage{verbatim} % adds environment for commenting out blocks of text & for better verbatim
\usepackage{subfig} % make it possible to include more than one captioned figure/table in a single float
\usepackage{caption}

\usepackage{makecell}

\usepackage[all]{xy}
\SelectTips{cm}{}
\usepackage{stmaryrd}

\usepackage{amsfonts,amsmath,amssymb,bbm,amsthm,amsbsy}

\usepackage{pdfpages}

\usepackage{mathtools}

\usepackage{titling}
\newcommand{\subtitle}[1]{%
  \posttitle{%
    \par\end{center}
    \begin{center}\Large#1\end{center}
    \vskip0.5em}%
}

\usepackage[pagebackref,%colorlinks,linkcolor=black,citecolor=black%
    ,pdfborder={0 0 0}
  ,urlcolor=black,a4paper,hypertexnames=false]{hyperref}

\usepackage{color}
\usepackage{pdfcolmk}

\makeatletter
\def\blfootnote{\xdef\@thefnmark{}\@footnotetext}
\makeatother
\date{\today%
    \protect\blfootnote{%\copyright{\ N.~Heuer, C.~L\"oh 2019}. 
    This work was supported by the CRC~1085 \emph{Higher Invariants} 
    (Universit\"at Regensburg, funded by the DFG).
    \\
    MSC~2010 classification: 20E05, 57M07, 57M20}}
\usepackage[nouppercase]{scrlayer-scrpage}

\automark[subsection]{section}          
\deftripstyle{myheadings}%{\headmark}{}{\pagemark}%
             {\ifthenelse{\isodd{\value{page}}}{\leftmark}{\pagemark}}
             {}
             {\ifthenelse{\isodd{\value{page}}}{\pagemark}{\rightmark}}% 
             {}{}{}% 
\def\varepsilon{\epsilon}
\def\longrightarrow{\to}
\def\longmapsto{\mapsto}

\def\genrel#1#2{%
  \langle #1 \mid #2 \rangle}
\def\args{\;\cdot\;}

\makeatletter							%allows to repeat theorems
\newtheorem*{rep@theorem}{\rep@title}
\newcommand{\newreptheorem}[2]{%
\newenvironment{rep#1}[1]{%
 \def\rep@title{#2 \ref{##1}}%
 \begin{rep@theorem}}%
 {\end{rep@theorem}}}
\makeatother

\newtheorem{thm}{Theorem}[section]
\newtheorem{lemma}[thm]{Lemma}
\newtheorem{prop}[thm]{Proposition}
\newtheorem{corr}[thm]{Corollary}
\newtheorem{claim}[thm]{Claim}

\newtheorem*{thm*}{Theorem}
\newtheorem*{lemma*}{Lemma}
\newtheorem*{prop*}{Proposition}
\newtheorem*{corr*}{Corrolary}
\newtheorem*{claim*}{Claim}

\theoremstyle{definition}%{remark}
\newtheorem{rmk}[thm]{Remark}

\newtheorem*{rmk*}{Remark}
\newtheorem*{conj*}{Conjecture}

\theoremstyle{definition}
\newtheorem{defn}[thm]{Definition}
\newtheorem{exmp}[thm]{Example}
\newtheorem{setup}[thm]{Setup}
\newtheorem{quest}[thm]{Question}

\newtheorem*{defn*}{Definition}
\newtheorem*{exmp*}{Example}
\newtheorem*{quest*}{Question}
\newtheorem*{keyex*}{Key Example}

\newreptheorem{theorem}{Theorem}
\newreptheorem{corollary}{Corollary}
\newreptheorem{proposition}{Proposition}

\theoremstyle{theorem}
\newtheorem{theorem}{Theorem}

\theoremstyle{definition}
\newtheorem{example}[theorem]{Example}
\newtheorem{corollary}[theorem]{Corollary}
\usepackage{tikz}
\usepackage{tikz-cd}

\usepackage{IEEEtrantools}

\newenvironment{equ*}[1]{\begin{IEEEeqnarray*}{#1}}{\end{IEEEeqnarray*}}

\usepackage[nottoc,notlof,notlot]{tocbibind} % Put the bibliography in the ToC
\usepackage[titles,subfigure]{tocloft} % Alter the style of the Table of Contents

\newcommand{\R}{\mathbb{R}}

\newcommand{\Q}{\mathbb{Q}}
\newcommand{\Z}{\mathbb{Z}}
\newcommand{\N}{\mathbb{N}}

\DeclareMathOperator{\scl}{scl}
\DeclareMathOperator{\cl}{cl}
\newcommand{\Scl}{\mathcal{S}}

\newcommand{\xtt}{\texttt{x}}

\newcommand{\att}{\normalfont\texttt{a}}
\newcommand{\btt}{\normalfont\texttt{b}}

\newcommand{\lall}{\normalfont\texttt{lallop}}

\newcommand{\Dcl}{\mathcal{D}}

\newcommand{\genus}{\mathrm{genus}}

\DeclareMathOperator{\A}{A}
\DeclareMathOperator{\B}{B}
\DeclareMathOperator{\Rec}{R}
\DeclareMathOperator{\Ver}{V}
\DeclareMathOperator{\BP}{BP}
\DeclareMathOperator{\RP}{RP}
\DeclareMathOperator{\TP}{TP}
\DeclareMathOperator{\OTP}{OTP}
\DeclareMathOperator{\DOTP}{DOTP}
\def\otp#1#2#3{%
  [#1, (#2,#3)]}
\def\otpp#1#2#3{%
  [(#3,#2),#1]}
\def\dotp#1#2#3{%
  \llbracket (#1,#2), (#1, #3)\rrbracket}

\def\gvertx#1{%
  \fill #1 circle (0.1);}
\def\orarrow#1{%
  \begin{scope}[shift={#1}]
    \draw[->] (0:0.35) arc (0:330:0.35); 
  \end{scope}}

\DeclareMathOperator{\Deltal}{\Delta_{\normalfont{\texttt{l}}}}

\newcommand{\col}{\colon}

\def\qand{\quad\text{and}\quad}

\def\fa#1{%
  \forall_{#1}\;\;\;}
\def\exi#1{%
  \exists_{#1}\;\;\;}

\title{Simplicial volume of one-relator groups\\ and\\ stable commutator length}
\author{Nicolaus Heuer, Clara L\"oh}

\begin{document}

\maketitle

\begin{abstract}
A one-relator group is a group $G_r$ that admits a presentation
$\genrel S r$ with a single relation $r$.  One-relator groups form a
rich classically studied class of groups in Geometric Group Theory.
If $r \in F(S)'$, we introduce a \emph{simplicial volume $\| G_r \|$
  for one-relator groups}.  We relate this invariant to the
\emph{stable commutator length} $\scl_S(r)$ of the element $r \in
F(S)$.  We show that often (though not always) the linear relationship
$\| G_r \| = 4 \cdot \scl_S(r) - 2$ holds and that every rational
number modulo~$1$ is the simplicial volume of a one-relator group.

Moreover, we show that this relationship holds approximately for proper powers
and for elements satisfying the small cancellation condition~$C'(1/N)$, with a
multiplicative error of $O(1/N)$. 
This allows us to prove for random elements of~$F(S)'$ of length~$n$
that~$\| G_r \|$ is $2 \log(2 |S| - 1)/3 \cdot n / \log(n) + o(n/\log(n))$
with high probability, using an analogous result of Calegari--Walker
for stable commutator length.
\end{abstract}

%%%%%%%%%%%%%%%%%%%%%%%%%%%%%%%%%

%%%%%%%%%%%%%%%%%%%%%%%%%%%%%%%%%%%%%%%%%%%%%%%%%%%%%%%%%%%%%%%%%
\section{Introduction}

A \emph{one-relator group} is a group $G_r$ that admits a presentation
$\genrel S r$ with a single relation~$r \in F(S)$.  This rich and well
studied class of groups in Geometric Group Theory generalises surface
groups and shares many properties with them.

A common theme is to relate the geometric properties of a classifying
space of~$G_r$ to the algebraic properties of the relator~$r \in
F(S)$.  For example, $r \in F(S)'\setminus \{e\}$ if and only if
$H_2(G_r; \Z) \not \cong 0$. In this case $H_2(G_r ; \Z)$ is infinite cyclic and
generated by the \emph{fundamental class} $\alpha_r \in H_2(G_r ;
\Z)$. We define the \emph{simplicial volume} of~$G_r$ as $\| G_r \| :=
\| \alpha_{r,\R} \|_1$, the $l^1$-semi-norm of the fundamental class
$\alpha_r$ (Section~\ref{subsec:defonerelsimvol}).

For every element $w \in F(S)'$, the \emph{commutator
  length~$\cl_S(w)$ of $w$ in $F(S)$} is defined via
\[
\cl_S(w) := \min \bigl\{ n \in \N
                \bigm| \exi{g_1,\dots, g_n, h_1,\dots, h_n \in F(S)} w = [g_1,h_1] \cdots [g_n,h_n] \bigr\}
\]
and the \emph{stable commutator length} is the limit
\[
\scl_S(w) := \lim_{n \to \infty} \frac{\cl_S(w^n)}{n}.
\]
The study of stable commutator length has seen much progress in recent
years by Calegari and others~\cite{Calegari,Calegari_sss,CalegariFujiwara}.  Calegari
showed that in a non-abelian free group, stable commutator length is
always rational and computable in polynomial time with respect to the
word length~\cite{Calegari_rational}. Moreover, it is known that there
is a \emph{gap} of~$1/2$ in the stable commutator length, i.e., if $w
\in F(S)' \setminus \{e\}$, then $\scl_S(w) \geq
1/2$~\cite{DuncanHowie}.

The theme of this paper is to connect the (topological) invariant~$\|
G_r \|$ to the (algebraic) invariant~$\scl_S(r)$.  The motivating
example is the following:

\begin{keyex*}[surface groups]\label{exa:key}
  Let $g \in \N_{>0}$, set $S_g := \{ a_1,\dots, a_g, b_1, \dots, b_g
  \}$ and let $ r_g := [a_1,b_1] \cdot \dots \cdot [a_g, b_g]$ in
  $F(S_g)$. Then $\| \genrel {S_g} {r_g} \| = 4g-4$~\cite[p.~9]{vbc}
  and $\scl_{S_g} (r_g) = g - \frac{1}{2}$~\cite[Theorem~2.93,
    Theorem~2.101]{Calegari}.  Therefore in this case, stable commutator length and
  simplicial volume are related by the formula
 $$
  \bigl\|\genrel{S_g}{r_g}\bigr\|
     = 4 \cdot \Bigl( \scl_{S_g}(r_g) - \frac12 \Bigr).
 $$
\end{keyex*}

We see that the relationship $\| G_r \| = 4 \cdot \scl_S(r) - 2$ holds
in many instances, though \emph{not} always:

\begin{exmp}[Example \ref{exmp:counterexmp}] \label{exmp:counterex_introduction}
  The element $v = \texttt{aaaabABAbaBAAbAB} \in F(\{ \att, \btt \})$,
  where $\texttt{A}=\texttt{a}^{-1}$ and $\texttt{B} = \texttt{b}^{-1}$, 
satisfies that $\scl_{\{\att,\btt\}}(v) = 5/8$, but $\| G_v \| = 0$.
\end{exmp}

Thus we ask:
\begin{quest}\label{q:main}
  Let $S$ be a set and let $r \in F(S)'$ be non-trivial.
  When is it true that
  \[ \|G_r\|
     %= \bigl\| \genrel S r \bigr\|
     = 4 \cdot \Bigl(\scl_S (r) - \frac12\Bigr)
  \text{ ?}
  \]
\end{quest}
Observe that the right-hand side is always non-negative because of the
$1/2$-gap of stable commutator length in free groups~\cite{DuncanHowie}.

As seen in Example \ref{exmp:counterex_introduction} there are
elements $r \in F(S)'$ where $\| G_r \| \geq 4 \cdot \scl_S(r) - 2$
fails to hold.  We do not know if the other inequality always
holds. We are only able to obtain a weaker strict inequality $\| G_r
\| < 4 \cdot \scl_S(r)$ (Corollary~\ref{cor:weakupper}).

In this article, we find a positive answer to Question \ref{q:main} in
various instances.  There are several ways to compute stable
commutator length. In order to prove the results of this article we
will make these interpretations available also for the simplicial
volume of one-relator groups.  These will be:
\begin{itemize}
\item \emph{topologically}, in terms of surfaces
  (Proposition~\ref{prop:svsurface}),
\item \emph{algebraically}, in terms of commutator lengths
  (Corollary~\ref{cor:algsimvol}), 
\item \emph{dually}, in terms of quasimorphisms
  (Proposition~\ref{prop:svviaqm}).
\item \emph{combinatorially/algorithmically}, in terms of van Kampen diagrams on surfaces
  (Proposition~\ref{prop:admissible maps and van Kampen diagrams}),
\end{itemize}

\subsection*{Decomposable relators}

\begin{theorem}[decomposable relators; Section~\ref{subsec:decomprel}]\label{introthm:decomprel}
  The answer to Question~\ref{q:main} is positive in the following cases:
  \begin{enumerate}
  \item $S = S_1 \cup S_2$ with $S_1 \cap S_2 = \emptyset$,
    and $r= r_1 r_2$, where $r_1 \in F(S_1)', r_2 \in F(S_2)'$ are non-trivial;
    %and
  \item $S = S' \cup {t}$ and $r= r_1 t r_2 t^{-1}$ with $t \not \in S$, $r_1,r_2 \in F(S')'\setminus\{e\}$.
  \end{enumerate}
\end{theorem}

In previous
work, we combined similar calculations over more general base groups
with known values of stable commutator length, to manufacture closed
$4$-manifolds with arbitrary rational simplicial
volume~\cite{heuerloeh4mfd} or with arbitrarily small transcendental
simplicial volume~\cite{heuerloehtrans}.

Using results of Calegari~\cite{Calegari_sss}, Theorem~\ref{introthm:decomprel}
implies:

\begin{corollary}
For every rational number $0 \leq q < 1$ there is one-relator group~$G_r$ with $\| G_r \| \equiv q \mod 1$. 
\end{corollary}

We do not know if there are one-relator groups with irrational simplicial volume.

\subsection*{Hyperbolic one-relator groups}

It is well-known that one-relator groups are hyperbolic if the relator
$r$ is a proper power $r'^N$ or if the relator satisfies a small
cancellation condition. We obtain an affirmative answer to Question~\ref{q:main}
in those cases, up to multiplicative constants of size~$O(N^{-1})$.

\begin{theorem}[small cancellation elements, Theorem~\ref{thm:small cancellation}]
  \label{introprop:smallcanc}
Let $r \in F(S)'$ be an element that satisfies the small cancellation
condition~$C'(1/N)$ for some $N \geq 6$. Then
$$
4 \cdot \scl_S(r)
> \| G_r \|
\geq \Bigl(1 - \frac 6 N \Bigr) \cdot 4 \cdot \scl_S(r). 
$$  
\end{theorem}

\begin{theorem}[proper powers, Theorem~\ref{thm:powers}]\label{introthm:powers}
  If $r = r'^{N}$ for some~$r' \in F(S)'\setminus\{e\}$ and~$N > 6$, then
  \[
  4 \cdot \scl_S(r) > \| G_r \| \geq \Bigl( 1 - \frac{6}{N} \Bigr) \cdot 4 \cdot \scl_S(r).
  \]
  In particular, we have that
  \[
  \lim_{N \to \infty} \frac{\| G_{r^N} \|}{N} = 4 \cdot \scl_S(r).
  \]
\end{theorem}

Using Theorem \ref{introprop:smallcanc} and a result by
Calegari--Walker \cite{rand_rigid} we are able to compute the
distribution of the simplicial volume of random one-relator groups:

\begin{theorem}[Theorem \ref{thm: random simvol of elts}]
Fix a set $S$ and let $r \in F(S)$ be a random reduced element of even
length $n$, conditioned to lie in the commutator subgroup
$F(S)'$. Then for every $\epsilon > 0$ and $C > 1$,
$$
\left| \|G_r \| \cdot \frac{\log(n)}{n}  - \frac{2\cdot \log(2 |S| -1)}{3} \right| \leq \epsilon
$$
with probability $1-O(n^{-C})$.
\end{theorem}

\subsection*{Simplicial volume via linear programming}

Calegari showed that $\scl_S(r)$ may be computed in polynomial time in
$|r|$ by reducing it to a linear programming problem
\cite{Calegari_rational}. This revealed that $\scl_S(r)$ is in
particular rational. The corresponding algorithm ($\texttt{scallop}$)
has been implemented and is open-source available \cite{scallop}.  We
are not able to reduce the computation of $\| G_r \|$ to a similar
programming problem. However, we introduce a new invariant
$\texttt{lallop}(r)$ (Definition \ref{defn:lall}) to bound $\| G_r \|$
from below.  We show that $\texttt{lallop}$ can be computed by
reducing it to a linear programming problem, we implemented
this algorithm~\cite{lallop,lallophs}, and used this lower bound
effectively.

\begin{theorem}[$\lall$; Theorem~\ref{thm:lallop}]\label{introthm:lallop} 
  Let $S$ be a set and $r \in F(S)'\setminus\{e\}$.
  Then
  \begin{align*}
    %0 \leq
    \lall(r) & \leq \| G_r \|, \\
    \lall(r) & \leq 4 \cdot \Bigl( \scl_S(r) - \frac{1}{2} \Bigr),
  \end{align*}
  and there is an algorithm to compute~$\lall(r)$ that is polynomial
  in the word length of~$r$ over~$S$.  Moreover, $\lall(r) \in \Q$.
\end{theorem}

In this way we may estimate $\| G_r \|$ explicitly, which sometimes
allows us to compute $\| G_r \|$ also for non-decomposable relators.

\begin{example}[Proposition~\ref{prop:example}] \label{exmp:main example intro}
  Let $m \in \N_{\geq 2}$ and $r_m= [\att, \btt] [\att, \btt^{-m}] \in
  F(\{ \att, \btt \})$. Then
  $$
  \| G_{r_m} \| \leq  \frac{2m-4}{m-1} = 4 \cdot \Bigl( \scl_{\{ \att, \btt \}}(r_m) - \frac{1}{2}\Bigr).
  $$
  For $m \in \{2,3,4 \}$ we compute that $\lall(r_2)=0$, $\lall(r_3)=1$ and $\lall(r_4)=\frac{4}{3}$. Thus
  $$
  \fa{m\in \{2,3,4\}}
  \| G_{r_m} \| = \frac{2m-4}{m-1} = 4 \cdot \Bigl( \scl_{\{ \att, \btt \}}(r_m) - \frac{1}{2}\Bigr).
  $$
\end{example}

\subsection*{Follow-up questions}

Combining Question~\ref{q:main} with known properties of stable
commutator length and simplicial volume raises these follow-up
questions:

\begin{quest}\label{q:followup}
  \hfil
  \begin{enumerate}
  \item Let $S$, $S'$ be sets and let $r \in F(S)'\setminus\{e\}$, $r'
    \in F(S')'\setminus\{e\}$ be relators with~$\genrel S r \cong
    \genrel {S'} {r'}$. Does this imply that $\scl_S r = \scl_{S'}
    r'$\;?
  \item
    Is the simplicial volume of one-relator groups computable?
  \item 
    Is there a gap~$C > 0$ such that for every set~$S$ and every
    relator~$r \in F(S)'\setminus \{e\}$ either $\| G_r \| = 0 $ or $\|
    G_r \| \geq C$\;?
  \item Louder and Wilton \cite{LouderWilton} showed that much of the
    geometry of one-relator groups with defining relation $r$ may be controlled by the
    \emph{primitivity rank}, denoted by $\pi(r)$. From their computations it is apparent
    that if $\pi(r) > 2$ then $\scl_S(r) > 1/2$. Is there a similar
    connection to the simplicial volume?
  \end{enumerate}
\end{quest}

%%%%%%%%%%%%%%
\subsection*{Organisation of this article}

We first recall simplicial volume of manifolds as well as stable
commutator length (Section~\ref{sec:prelims}). We then introduce
simplicial volume of one-relator groups (Section~\ref{sec:svonerel})
and establish some basic properties
(Theorem~\ref{introthm:decomprel}). In Section~\ref{sec:vKdiag}, we
describe simplicial volume of one-relator groups in terms of
van~Kampen diagrams, leading to a proof of
Proposition~\ref{introprop:smallcanc} and
Theorem~\ref{introthm:powers}. The analysis of simplicial volume of
random one-relator groups is carried out in Section~\ref{sec:random}.
In Section~\ref{sec:lallop}, we introduce the computational
invariant~$\lall$ and prove Theorem~\ref{introthm:lallop}; moreover,
we include a sample computation (Example~\ref{exmp:main example
  intro}).

%%%%%%%%%%%%%%%%
\subsection*{Acknowledgements}

The first author would like to thank Martin Bridson for his support
and many very helpful discussions. 
We would further like to thank the referee for many helpful suggestions which significantly improved both the content and the style of the paper.

%%%%%%%%%%%%%%%%%%%%%%%%%%%%%%%%%%%%%%%%%%%%%%%%%%%%%%%%%%%%%%%
\section{Preliminaries}\label{sec:prelims}

We summarise notation and basic properties of simplicial volume and
stable commutator length.

%%%%%%%%%%%%%%%%%%%%%%%%%%%%%%%%%%%%%%%%
\subsection{Simplicial volume}\label{sec:simvoldef}

We quickly recall the notion of simplicial volume of manifolds, which
is based on the $l^1$-semi-norm on singular homology. Let $X$ be a
topological space and let $d \in \N$. Then the \emph{$l^1$-semi-norm}
on~$H_d(X;\R)$ is defined as
\begin{align*}
  \| \cdot \|_1 \colon H_d(X;\R) & \longrightarrow \R_{\geq 0}
  \\
  \alpha & \longmapsto \inf \bigl\{ |c|_1 \bigm| c \in
  C_d(X;\R),\ \partial c = 0,\ [c] = \alpha \bigr\};
\end{align*}
here, $C_d(X;\R)$ is the singular chain module of~$X$ in degree~$d$
with $\R$-coefficients and $|\cdot|_1$ denotes the $l^1$-norm
on~$C_d(X;\R)$ associated with the basis of singular simplices.

\begin{defn}[simplicial volume~\cite{vbc}]
  Let $M$ be an oriented closed connected $d$-dimensional manifold.
  Then the \emph{simplicial volume of~$M$} is 
  \[ \|M\| := \bigl\| [M]_\R \bigr\|_1,
  \]
  where $[M]_\R \in H_d(M;\R)$ is the $\R$-fundamental class
  of~$M$.
\end{defn}

On the one hand, simplicial volume is a homotopy invariant
of (oriented) compact manifolds that is compatible with mapping
degrees: If $f \colon M_1 \longrightarrow M_2$ is a continuous map
between oriented closed connected manifolds of the same dimension,
then
\[ \mathopen|\deg f | \cdot \| M_2 \| \leq \| M_1 \|.
\]
On the other hand, simplicial volume is related in a non-trivial way
to Riemannian volume, e.g., in the presence of enough negative
curvature~\cite{vbc,inoueyano,thurstonln,lafontschmidt,connellwang,mapsimvol}.
A very different source of manifolds with non-zero simplicial volumes
are our constructions via stable commutator
length~\cite{heuerloeh4mfd}.

Dually, we can describe the $l^1$-semi-norm (and whence simplicial volume)
in terms of bounded cohomology~$H^*_b(\args;\R)$:

\begin{prop}[duality principle for the $l^1$-semi-norm~\protect{\cite[p.~6/7]{vbc}\cite[Lemma~6.1]{Frigerio}}] \label{prop:duality simvol bc}
  Let $X$ be a topological space, let $d \in \N$, and let $\alpha \in H_d(X;\R)$.
  Then
  \[
  \| \alpha \|_1  =  \sup \bigl\{ \langle \beta , \alpha \rangle
                          \bigm| \beta \in H^d_b(X, \R), \| \beta \|_\infty \leq 1
                          \bigr\}.
  \]
\end{prop}

\begin{corr}[duality principle for simplicial volume~\protect{\cite[p.~7]{vbc}}]
  Let $M$ be an oriented closed connected $d$-manifold. Then
  \[ \| M\| = \frac1{\|\varphi\|_\infty},
  \]
  where~$\varphi \in H^d(M;\R)$ is the singular cohomology class
  satisfying~$\langle \varphi, [M]_\R\rangle = 1$.
\end{corr}

%%%%%%%%%%%%%%%%%%%%%%%%%%%%%%%%%%%%
\subsection{Stable commutator length}\label{subsec:scldef}

In this section we give a very brief introduction to stable commutator
length. The main reference is Calegari's book~\cite{Calegari}.  For a
group $G$ let $G'$ be its commutator subgroup and let $g \in G'$. We
define the \emph{commutator length} $\cl_G(g)$ of an element $g \in
G'$ via
$$
\cl_G(g) := \min \bigl\{ n \in \N 
                 \bigm| \exi{x_1, \dots, x_n, y_1, \dots, y_n \in G} g = [x_1, y_1] \cdots [x_n, y_n]
                 \bigr\}.
$$
It is easy to see that commutator length is invariant under
automorphisms, in particular conjugations.

It will be convenient to extend the notion of commutator length to
``sums'' of group elements. If $m \in \N$ and $g_1, \dots, g_m \in G$
with~$g_1 \cdots g_m \in G'$, then one writes
\[ \cl_G(g_1 + \dots + g_m)
   := \min_{t_1, \dots, t_m \in G} \cl_G(t_1 g_1 t_1^{-1} \cdots t_mg_mt_m^{-1}).
\]
It is not hard to see that, as the notation suggests, the
value~$\cl_G(g_1 + \dots + g_m)$ is independent of the order of~$g_1,
\dots, g_m$.

%Moreover, we get the inequalities
%
%for all elements $g \in G'$ and formal sums $\sum_{i=1}^n g_i$ as above.

\begin{defn}[stable commutator length]
  Let $G$ be a group, let $m \in \N$, and let $g_1, \dots, g_m \in G$
  with~$g_1 \cdots g_m \in G'$.  The \emph{stable commutator length}
  of the tuple~$(g_1,\dots, g_m)$ is defined via
  $$
  \scl_G(g_1 + \cdots + g_m) := \lim_{n \to \infty} \frac{\cl_G(g_1^n + \cdots + g_m^n)}{n}.
  $$
\end{defn}

This limit indeed exists and \emph{stable} commutator length
has the following additive behaviour~\cite[Chapter~2.6]{Calegari}:
For all~$n \in \N_{>0}$ and all~$g\in G'$, we have
\begin{eqnarray*}
  \scl_G(n \cdot g) & = & \scl_G(g^n); %  \mbox{, and} \\
\end{eqnarray*}
For all~$g\in G$, $m \in \N$, and all~$g_1,\dots, g_m \in G$ with~$g_1 \cdots g_m \in G'$,
we have
\begin{eqnarray*}
\scl_G\biggl(g + g^{-1} + \sum_{i=1}^m g_i\biggr) & = & \scl_G\biggl(\sum_{i=1}^m g_i\biggr).
\end{eqnarray*}

\subsubsection{(Stable) Commutator length in free groups via surfaces}

Commutator length and stable commutator length have a geometric
interpretation.  For what follows, we will restrict our attention to
(stable) commutator length of the free group $F(S)$ with generating
set $S$, even though every result in this section holds for general
groups.

Let $m \in \N$ and let $g_1, \ldots, g_m \in F(S)$ be elements such
that $g_1 \cdots g_m \in F(S)'$.  Let $B_S$ be a bouquet of $|S|$
circles labelled by the elements of $S$; we identify $F(S)$ with
$\pi_1(B_S)$ in the canonical way. Moreover, let $\gamma_1, \ldots,
\gamma_m \col S^1 \to X$ be based loops in $B_S$ such that
$[\gamma_i]_* = g_i$ in $F(S)$.

\begin{defn}[$\cl$- and $\scl$-admissible maps] \label{defn:scl admissible maps}
  Let $\Sigma$ be an orientable surface with boundary $\partial
  \Sigma$, with genus at least $1$ and with the inclusion map $\iota
  \col \partial \Sigma \to \Sigma$. Moreover, let $f \col \Sigma \to
  B_S$ be a map from $\Sigma$ to $B_S$ and let $\partial f \col \partial \Sigma
  \to \coprod_{i=1}^m S^1$ be the restriction of $f$ to the boundary such that the diagram
  $$
  \begin{tikzcd}
    \partial \Sigma \arrow[r, "\iota"] \arrow[d, "\partial f"] & \Sigma \arrow[d, "f"] \\
    \coprod_{i=1}^m S^1 \arrow[r, "{\gamma_1, \ldots, \gamma_m}"] & X
  \end{tikzcd}
  $$
  commutes. 
  We say that the pair $(f, \Sigma)$ is
  \begin{itemize}
  \item \emph{$\cl$-admissible to $g_1+ \cdots + g_m$}, if $\partial
    f$ is a degree~$1$ map on all components and
  \item \emph{$\scl$-admissible to $g_1+ \cdots + g_m$}, if there is
    an integer $n(\Sigma, f) \in \Z$, called the \emph{degree of
      $(\Sigma,f)$}, such that $H_1(\partial f;\Z) [\partial \Sigma] =
    n(\Sigma, f) \cdot [\coprod_{i=1}^m S^1]$ in $H_1(\coprod_{i=1}^m
    S^1; \Z)$.
\end{itemize}
\end{defn}

The ``set'' of all $\cl$- and $\scl$-admissible pairs $(f,\Sigma)$ to
the formal sum $g_1+\dots+g_m$ will be denoted by
$\Sigma^{\cl}_{\partial}(g_1+\dots+g_m)$ and
$\Sigma_{\partial}(g_1+\dots+g_m)$, respectively (strictly speaking,
this set is a class, but we could fix models for each homeomorphism
type of surfaces to turn this into an actual set).

\begin{prop}[(stable) commutator length via surfaces~\protect{\cite[Proposition~2.74]{Calegari}}] \label{prop:stable commutator length via surfaces}
  Let $S$ be a set, let $m \in \N$, and let $g_1, \dots, g_m \in F(S)$
  with~$g_1 \cdots g_m \in F(S)'$.  Then
  \begin{eqnarray*}
    \cl_{F(S)}(g_1+\cdots+ g_m) &=& \min_{(f,\Sigma) \in \Sigma^{\cl}_{\partial}(g_1+\dots+g_m)}  \genus(\Sigma) \mbox{, and} \\
    \scl_{F(S)}(g_1+\cdots+ g_m) &=& \inf_{(f,\Sigma) \in \Sigma_{\partial}(g_1+\dots+g_m)}  \frac{-\chi^-(\Sigma)}{2 \cdot n(f,\Sigma)}.
  \end{eqnarray*}
\end{prop}
Here, $\chi^-$ denotes the Euler characteristic that ignores spheres and disks. I.e., if $\Sigma = \bigsqcup_{i=1}^n \Sigma_i$ with connected components~$\Sigma_i$, then we define
$$
\chi^-(\Sigma) = \sum_{i=1}^n \min \bigl( 0, \chi(\Sigma_i) \bigr).
$$  
To shorten notation we will frequently simply write $\cl_S$ and $\scl_S$ instead of $\cl_{F(S)}$ and $\scl_{F(S)}$.

\subsubsection{Stable commutator length via quasimorphisms}

Let $G$ be a group. A map $\phi \col G \to \R$ is called a
\emph{quasimorphism} if there is a constant $C > 0$ such that
$$
\sup_{g,h \in G}
| \phi(g) + \phi(h) - \phi(gh) | \leq C.
$$
The smallest such bound~$C$ is called the \emph{defect of $\phi$} and
is denoted by $D(\phi)$.  If $\phi$ is a linear combination of a
bounded function and a homomorphism, then $\phi$ is called a
\emph{trivial} quasimorphism.  Quasimorphisms are intimately related
to~$H^2_b(G; \R)$, the bounded cohomology of $G$ in degree~$2$ with
trivial real coefficients: The boundary of a quasimorphism $\delta^1
\phi$ defines a non-trivial class in~$H^2_b(G,\R)$ if and only if
$\phi$ is non-trivial. Moreover, all exact classes
in~$H^2_b(G,\R)$ arise in this way~\cite[Theorem~2.50]{Calegari}.

A quasimorphism $\phi \col G \to \R$ is called \emph{homogeneous}, if
for all $g \in G$, $n \in \Z$ we have that $\phi(g^n) = n \cdot
\phi(g)$. The set of all homogeneous quasimorphisms on~$G$ is denoted
by~$Q^h(G)$.  Stable commutator length may be computed via
quasimorphisms using \emph{Bavard's duality theorem} proved by Bavard
and generalised by Calegari:

\begin{thm}[Bavard duality~\cite{Bavard}\protect{\cite[Theorem~2.79]{Calegari}}] \label{thm:Bavard}
  Let $G$ be a group, let $m \in \N$, and let $g_1, \ldots, g_m \in G$
  such that $g_1 \cdots g_m \in G'$. Then
  $$
  \scl_G(g_1+ \cdots + g_m) = \sup_{\phi \in Q^h(G)} \frac{\sum_{i=1}^m \phi(g_i)}{2 \cdot D(\phi)}.
  $$
\end{thm}

%%%%%%%%%%%%%%%%%%%%%%%%%%%%%%%%%%%%%%%%%%%%%%%%%%%%%%%%%%%%%%%%%%%%%%%%%
\section{Simplicial volume of one-relator groups}\label{sec:svonerel}

We introduce the simplicial volume of one-relator presentations and
one-relator groups and establish basic properties as well as
alternative descriptions (via surfaces, commutator length, and
quasimorphisms).

%%%%%%%%%%%%%%%%%%
\subsection{Setup and notation}\label{subsec:defonerelsimvol}

\begin{setup}\label{setup:onerel}
  Let $F(S)$ be the free group on some alphabet~$S$, let $r \in F(S)'$
  be a non-trivial element in the commutator subgroup, and let $G_r :=
  \genrel Sr$ be the one-relator group defined by the
  presentation~$(S,r)$.
  
  We write~$P_r$ for the presentation complex of~$G_r$ associated with
  the presentation~$(S,r)$ and $X_r$ for a model of the classifying
  space of~$G_r$ obtained by attaching higher-dimensional cells
  to~$P_r$. Let $c_r \colon P_r \longrightarrow X_r$ be the inclusion
  map.  Because $r$ is in the commutator subgroup, the $2$-cell
  of~$P_r$ defines a homology class~$\widetilde \alpha_r \in
  H_2(P_r;\Z)$.
\end{setup}
  
\begin{defn}[fundamental class, simplicial volume of a one-relator presentaion]
  \label{defn:simvol one relator groups}
  In the situation of Setup~\ref{setup:onerel}, we define:
  \begin{itemize}
  \item The \emph{fundamental class of~$(S,r)$}:
    \[ \alpha_r := H_2(c_r;\Z)(\widetilde \alpha_r) \in H_2(G_r;\Z).
    \]
  \item The \emph{$\R$-fundamental class~$\alpha_{r,\R} \in
    H_2(G_r;\R)$ of~$(S,r)$} as the image of~$\alpha_r$ under the
    change of coefficients map~$H_2(G_r;\Z) \longrightarrow
    H_2(G_r;\R)$.
  \item The \emph{simplicial volume of~$(S,r)$}:
    \[ \| (S,r) \| := \| \alpha_{r,\R} \|_1 \in \R_{\geq 0}.
    \]
    Here, $\|\cdot\|_1$ denotes the $l^1$-semi-norm on singular
    homology~$H_*(\args;\R)$.
  \end{itemize}
\end{defn}

\begin{rmk}[simplicial volume of one-relator groups]\label{rem:svonereldef}
  In the situation of Setup~\ref{setup:onerel}, the Hopf
  formula~\cite[Theorem~II.5.3]{brown} shows that $H_2(G_r;\Z)$ is
  isomorphic to~$\Z$ and that $\alpha_r$ is a generator
  of~$H_2(G_r;\Z)$. In particular: If $(S', r')$ is another
  one-relator presentation of~$G_r$ with~$r' \in F(S')'$, then
  $\alpha_{r'} \in \{\alpha_r, -\alpha_r\}$. Hence, the simplicial
  volume~$\|(S,r)\| = \| (S',r')\|$ depends only on the group and not
  on the chosen presentation. Therefore, we also write
  \[ \|G_r\| := \| (S,r)\|
  \]
  for the \emph{simplicial volume of the one-relator group~$G_r$}.
  
  Because $c_r \colon P_r \longrightarrow X_r$ is a
  $\pi_1$-isomorphism, the mapping theorem in bounded
  cohomology~\cite[p.~40]{vbc}\cite{ivanov}\cite[Theorem~5.9]{Frigerio}
  shows that
  \[ \| G_r\| = \|\alpha_{r,\R} \|_1 = \| \widetilde \alpha_{r,\R} \|_1.
  \]
  If $r$ is not a proper power, then $G_r$ is
  torsion-free~\cite{karrassmagnussolitar} and the presentation
  complex~$P_r$ already is a model of the classifying space
  of~$G_r$~\cite{cockcroft1954two}.
\end{rmk}

\begin{exmp}[hyperbolic groups and proper powers]\label{exa:hyp}
  If, in the situation of Setup~\ref{setup:onerel}, $G_r$ is
  hyperbolic, then because the class~$\alpha_{r,\R}$ is non-zero, it
  follows from Mineyev's non-vanishing result for bounded cohomology
  of hyperbolic groups~\cite[Theorem~15]{mineyev} and the duality
  principle (Proposition~\ref{prop:duality simvol bc}) that
  \[ \| G_r\| = \| \alpha_{r,\R} \|_1 > 0.
  \]
  For instance, whenever the relator~$r$ is a proper power, then $G_r$
  is a word-hyperbolic group. Newman's spelling theorem~\cite{newman}
  shows that Dehn's algorithm works in such groups.
\end{exmp}

\begin{exmp}[amenable case]
  In the situation of Setup~\ref{setup:onerel}, the group~$G_r$ is
  amenable if and only if $G_r \cong
  \Z^2$~\cite{ceccherinisilbersteingrigorchuk}.  Clearly, in this
  case, $P_r \simeq S^1 \times S^1$ and $\|G_r\| =0$.
\end{exmp}

%%%%%%%%%%%%%%%%%%
\subsection{Mapping degrees}

The simplicial volume of one-relator groups has the following
simple functoriality property with respect to group homomorphisms:

\begin{defn}[degree]
  Let $S_1$, $S_2$ be sets and let $r_1 \in F(S_1)'\setminus\{e\}$,
  $r_2 \in F(S_2)'\setminus\{e\}$.  If $f \colon G_{r_1} =
  \genrel{S_1}{r_1} \longrightarrow \genrel{S_2}{r_2} = G_{r_2}$ is a
  group homomorphism, then there is a unique integer~$\deg f$, the
  \emph{degree of~$f$}, with
  \[ H_2(f;\Z)(\alpha_{r_1}) = \deg f \cdot \alpha_{r_2} \in H_2(G_{r_2};\Z). 
  \]
  %(in the pathological case of~$r_2 = \varepsilon$, we set $\deg f := 1$).
\end{defn}

This notion of degree is a generalisation of the notion of degree for
maps between manifolds or for $l^1$-admissible maps in the sense of
Definition~\ref{def:admissiblel1}.  Strictly speaking, the sign of the
degree depends on the chosen one-relator presentation (and not only on
the one-relator group), but this will not cause any trouble.

\begin{prop}[functoriality]\label{prop:functoriality}
  Let $S_1$, $S_2$ be sets, let $r_1 \in F(S_1)'\setminus\{e\}$, $r_2 \in F(S_2)'\setminus\{e\}$,
  and let $f \colon G_{r_1} \longrightarrow G_{r_2}$ be a group homomorphism.
  Then
  \[ \bigl|\deg f\bigr| \cdot \|G_{r_2} \| \leq \|G_{r_1}\|.
  \]
\end{prop}
\begin{proof}
  We have~$H_2(f;\R)(\alpha_{r_1,\R}) = \deg f \cdot \alpha_{r_2,\R}$. Because
  $H_2(f;\R)$ does not increase~$\|\cdot\|_1$, the claim follows.
\end{proof}

\begin{exmp}
  Let $S$ be a set, let $r \in F(S)'\setminus\{e\}$, and let $N \in
  \N_{>0}$. Then the canonical homomorphism~$\genrel S {r^N}
  \longrightarrow \genrel S r$ has degree~$N$, and we obtain
  \[ \|G_r \| \leq \frac 1N \cdot \| G_{r^N}\|.
  \]
  Moreover, we will see that the limit~$\lim_{N \rightarrow \infty}
  1/N \cdot \|G_{r^N}\|$ is equal to~$\scl_S r$
  (Theorem~\ref{thm:powers}).
\end{exmp}

\begin{exmp}
  Let $S \subset \widetilde S$ be sets, let $r \in
  F(S)'\setminus\{e\}$, and let $\widetilde r \in F(\widetilde S)'$ be
  the corresponding element of~$F(\widetilde S)$. Then the two canonical
  group homomorphisms~$\genrel Sr \longrightarrow \genrel {\widetilde
    S}{\widetilde r}$ (given by the inclusion of~$S$ into~$\widetilde
  S$) and $\genrel {\widetilde S}{\widetilde r} \longrightarrow
  \genrel SR$ (given by projecting~$\widetilde S \setminus S$ to the
  neutral element) both have degree~$1$. Hence,
  \[ \| G_r \| = \| G_{\widetilde r}\|.
  \]
  In particular, omitting the generating set~$S$ in the notation~$\|G_r\|$
  is no real loss of information.
\end{exmp}

One-relator groups that satisfy the property in Question~\ref{q:main}
might inherit interesting properties for the stable commutator length
of the relator from the mapping degree functoriality of simplicial
volume (Proposition~\ref{prop:functoriality}).

%%%%%%%%%%%%%%
\subsection{Decomposable relators}\label{subsec:decomprel}

We will now compute the simplicial volume of one-relator groups with
decomposable relators, using the computation of the $l^1$-semi-norm in
degree~$2$ in these cases via the filling view and the calculation of
stable commutator length of decomposable
relators~\cite[Section~6.3]{heuerloeh4mfd}.  We only need to verify
that our current situation fits into that context.

\begin{proof}[Proof of Theorem~\ref{introthm:decomprel}]
  For the \emph{first part}, we let $S = S_1 \cup S_2$ with
  $S_1 \cap S_2 = \emptyset$ and $r = r_1 r_2$ with~$r_1 \in F(S_1)'\setminus \{e\}$,
  $r_2 \in F(S_2)'\setminus \{e\}$,
  and we note that
  \[ G_r = \genrel{S}{r}
  = (F(S_1) * F(S_2)) / \langle r_1 \cdot r_2\rangle^\triangleleft
  \cong F(S_1) *_\Z F(S_2),
  \]
  where the amalgamation homomorphisms~$\Z \longrightarrow F(S_1)$ and
  $\Z \longrightarrow F(S_2)$ are given by~$r_1$ and~$r_2$,
  respectively.  In order to use the previous computations for
  decomposable relators~\cite[Section~6.3]{heuerloeh4mfd}, we consider
  the double mapping cylinder
  \[ P := Z_1 \cup_{(z,1) \sim (\overline z, 1)} Z_2
  \]
  constructed by gluing the cylinders
  \begin{align*}
    Z_1 & := \Bigl(\bigvee_{S_1} S^1\Bigr) \cup_{\text{$r_1$ on~$S^1 \times \{0\}$}} \bigl(S^1 \times [0,1]\bigr)
    \\
    Z_2 & := \Bigl(\bigvee_{S_2} S^1\Bigr) \cup_{\text{$r_2$ on~$S^1 \times \{0\}$}} \bigl(S^1 \times [0,1]\bigr)
  \end{align*}
  Let $\widetilde \alpha \in H_2(P;\Z)$ be the canonical class
  constructed by gluing generators of~$H_2(Z_1, S^1 \times \{1\};\Z)
  \cong \Z$ and $H_2(Z_2, S^1 \times \{1\};\Z) \cong \Z$ and 
  let $c \colon P \longrightarrow BG_r$ be the classifying map. Then
  $H_2(c;\Z)(\widetilde\alpha)$ is a generator of~$H_2(G_r;\Z)$ and thus 
  \[ H_2(c;\Z) (\widetilde \alpha) = \pm \alpha_r \in H_2(G_r;\Z). 
  \]
  Therefore, the $\R$-version~$\alpha_\R \in H_2(P;\R)$
  of~$H_2(c;\Z)(\widetilde \alpha)$ satisfies
  \begin{align*}
    \| G_r \|
    & = \| \alpha_{r,\R} \|_1
      = \| \alpha_\R \|_1
    & \text{(Remark~\ref{rem:svonereldef})}
    \\
    & = 4 \cdot \Bigl( \scl_{S_1 \cup S_2} (r_1 \cdot r_2) - \frac12 \Bigr)
    & \text{\cite[Theorem~6.14]{heuerloeh4mfd}}
    \\
    & = 4 \cdot \Bigl( \scl_{S} r - \frac12\Bigr).
  \end{align*}
  
  For the \emph{second part}, we can argue similarly: Let $S = S' \cup
  \{ t \}$ and $r = r_1 t r_2 t^{-1}$ with~$t \not\in S'$ and $r_1, r_2 \in
  F(S')\setminus \{e\}$.  The canonical class in the second homology
  of
  \[ \Bigl(\bigvee_S S^1\Bigr) \cup_{r_1,r_2} \bigl( S^1 \times [0,1] \sqcup S^1 \times [0,1]\bigr)
  \]
  maps under the classifying map to the fundamental class~$\pm \alpha_r$. Hence,
  we obtain from the filling view~\cite[Theorem~6.14]{heuerloeh4mfd}
  \begin{align*}
    \| G_r \| = \| \alpha_{r,\R} \|_1
    = 4 \cdot \Bigl(\scl_{S' \cup \{t\}} (r_1 \cdot t \cdot r_2\cdot t^{-1}) - \frac12\Bigr)
    = 4 \cdot \Bigl(\scl_{F(S)} r - \frac12\Bigr),
  \end{align*}
  as claimed. 
\end{proof}

%%%%%%%%%%%%%%
\subsection{Simplicial volume via surfaces} \label{subsec:simvol via surfaces and tilings}

Analogously to Proposition \ref{prop:stable commutator length via
  surfaces} we will compute $\| G_r \|$ using admissible surfaces.

\begin{defn}[$l^1$-admissible map]\label{def:admissiblel1}
  In the situation of Setup~\ref{setup:onerel}, an
  \emph{$l^1$-ad\-missible map for~$(S,r)$} is a pair~$(f,\Sigma)$,
  consisting of an oriented \emph{closed} connected surface~$\Sigma$
  of genus at least~$1$ and a continuous map~$f \colon \Sigma
  \longrightarrow X_r$. The unique integer~$n(f,\Sigma)$ satisfying
  \[ H_2(f;\Z)[\Sigma]_\Z = n(f,\Sigma) \cdot \alpha_r \in H_2(G_r;\Z)
  \]
  is the \emph{degree} of~$(f,\Sigma)$. We write~$\Sigma(r)$ for the
  ``set'' of all $l^1$-admissible maps for~$r$. 
\end{defn}

\begin{prop}[simplicial volume via surfaces]\label{prop:svsurface}
  In the situation of Setup~\ref{setup:onerel}, we have
  \[ \| G_r \| = \inf_{(f,\Sigma) \in \Sigma(r)}
                 \frac{-2 \cdot \chi(\Sigma)}%
                      {\bigl|n(f,\Sigma)\bigr|}.
  \]
\end{prop}
\begin{proof}
  This is a special case of the fact that the $l^1$-semi-norm in
  degree~$2$ coincides with the surface
  semi-norm~\cite{bargeghys}\cite[Proposition~2.4]{crowleyloeh}.
\end{proof}

In the following, we will mainly use this surface description of the
simplicial volume. For example, Proposition~\ref{prop:svsurface}
implies a weak upper bound for simplicial volume of one-relator groups
and leads to a straightforward proof of a description of simplicial
volume of one-relator groups in terms of commutator lengths:

\begin{corr}[weak upper bound]\label{cor:weakupper}
  In the situation of Setup~\ref{setup:onerel}, we have
  \[ \|G_r\| < 4 \cdot \scl_S r.
  \]
\end{corr}
\begin{proof}
  Let $(f,\Sigma) \in \Sigma_\partial(r)$ be an extremal
  scl-admissible surface for~$r$; such a surface is known to
  exist~\cite[Theorem~4.24]{Calegari}, satisfies
  \[ \scl_S r = \frac{-\chi(\Sigma)}{2 \cdot n(f,\Sigma)}
  \]
  and has positive degree on every boundary. 
  We then consider the oriented closed connected surface~$\overline
  \Sigma$ obtained by gluing disks to the boundary components
  of~$\Sigma$. This adds at most $n(f, \Sigma)$ many disks to the surface $\Sigma$, and thus $-\chi( \overline \Sigma) \geq - \chi(\Sigma) - n(f, \Sigma)$. Since $\scl_S(r) \geq \frac{1}{2}$ \cite{DuncanHowie}, we see that $-\chi(\overline \Sigma) \geq 0$, which shows that $\overline \Sigma$ has genus at least $1$.
  
   Then $f$ extends to an $l^1$-admissible map~$\overline
  f \colon \overline \Sigma \longrightarrow X_r$, since the boundary loops of $f$ are trivial in $X_r$. The degree of this map satisfies
  \[ n(\overline f, \overline\Sigma) = n(f,\Sigma).
  \]
  By construction, $\chi(\overline \Sigma) > \chi(\Sigma)$, and from 
  Proposition~\ref{prop:svsurface} we obtain
  \[ \|G_r\| \leq \frac{-2 \cdot \chi(\overline \Sigma)}{\bigl|n(\overline f, \overline \Sigma)\bigr|}
  < \frac{-2 \cdot \chi(\Sigma)}{\bigl|n(f,\Sigma)\bigr|}
  =  4 \cdot \scl_S r.
  \qedhere
  \]
\end{proof}

\begin{corr}[algebraic description of simplicial volume]\label{cor:algsimvol}
  In the situation of Setup~\ref{setup:onerel}, we have
  \[ \| G_r \| =
  \inf_{(n,\varepsilon) \in E}
    4 \cdot \frac{\cl_S(r^{\varepsilon_1} + \dots + r^{\varepsilon_n}) - 1}
                 {|\varepsilon_1 + \dots + \varepsilon_n|},
  \]
  where
  $E := \bigl\{ (n,\varepsilon) \bigm| n \in \N_{>0},\ \varepsilon \in \{-1,1\}^n,\
                     \varepsilon_1 + \dots + \varepsilon_n \neq 0 \bigr\}.
  $
\end{corr}
\begin{proof}
  During this proof, we will abbreviate the right hand side of the
  claimed equality by~$c(r)$.
  We will first show that $\| G_r\| \leq c(r)$: Let $n \in \N_{>0}$,
  let $t_1,\dots, t_n \in F(S)$, let $\varepsilon_1,\dots,
  \varepsilon_n \in \{-1,1\}$ with $\sum_{j=1}^n \varepsilon_j \neq 0$,
  and let 
  \[ N := \cl_S(t_1 \cdot r^{\varepsilon_1} \cdot t_1^{-1} \cdot \dots \cdot
              t_n \cdot r^{\varepsilon_n} \cdot t_n^{-1}) \in \N.
  \]
  It should be noted that $\varepsilon_1 + \dots + \varepsilon_n \neq 0$
  implies that~$N > 0$ (because we work in the free group~$F(S)$). 
  Then there exist~$a_1, \dots, a_N, b_1, \dots, b_N \in F(S)$ such that
  \begin{align}
    t_1 \cdot r^{\varepsilon_1} \cdot t_1^{-1} \cdot \dots \cdot
  t_n \cdot r^{\varepsilon_n} \cdot t_n^{-1}
  = [a_1,b_1] \cdot \dots \cdot [a_N, b_N]
  \label{eq:commsurf}
  \end{align}
  holds in~$F(S)$. In particular, $[a_1,b_1] \cdot \dots \cdot
  [a_N,b_N]$ lies in the normal subgroup of~$F(S)$ generated by~$r$
  and we obtain a corresponding, well-defined, group homomorphism
  \[ \varphi \colon
     \genrel{a_1,\dots, a_N, b_1, \dots, b_N}{[a_1,b_1] \cdot \dots \cdot [a_N,b_N]}
     \longrightarrow G_r
  \]
  (given by mapping the generators to the corresponding elements
  in~$G_r$).  Passing to classifying spaces, we find a continuous
  map~$f \colon \Sigma_N \longrightarrow P_r$ with~$\pi_1(f) =
  \varphi$; more concretely, we can construct~$f$ as the cellular map
  that wraps the $2$-cell of the standard CW-model of~$\Sigma_N$
  around the $2$-cell of~$X_r$ according to the relation in
  Equation~\eqref{eq:commsurf}.  By construction, $(f,\Sigma_N)$ is an
  $l^1$-admissible map for~$r$ with
  \[ n(f,\Sigma) = \varepsilon_1 + \dots + \varepsilon_N.
  \]
  Applying Proposition~\ref{prop:svsurface} shows that
  \begin{align*}
    \|G_r\|
    & \leq \frac{-2 \cdot \chi(\Sigma_N)}{|\varepsilon_1 + \dots + \varepsilon_n|}
    %\\
    %&
    = 4 \cdot \frac{N-1}{|\varepsilon_1+\dots + \varepsilon_n|}
    \\
    & = 4 \cdot \frac{\cl_S(t_1 \cdot r^{\varepsilon_1} \cdot t_1^{-1} \cdot \dots \cdot
              t_n \cdot r^{\varepsilon_n} \cdot t_n^{-1}) -1}{|\varepsilon_1 +\dots+\varepsilon_n|}.
  \end{align*}
  Taking the infimum over the right hand side shows that $\|G_r\| \leq c(r)$.

  It remains to prove the converse inequality~$\|G_r\| \geq c(r)$:
  Again, we use the description of~$\|G_r\|$ in terms of
  $l^1$-admissible maps (Proposition~\ref{prop:svsurface}). Let
  $(f,\Sigma) \in \Sigma(r)$ with~$n(f,\Sigma) \neq 0$ and let $N$ denote 
  the genus of~$\Sigma$. Without loss of generality we may assume that $f$
  is cellular. Following the map induced by~$f$ on the $1$-skeleta, 
  we lift $\pi_1(f) \colon \pi_1(\Sigma)
  \longrightarrow G_r$ to a homomorphism~$\psi \colon
  F(a_1,\dots, a_N, b_1, \dots, b_N) \longrightarrow F(S)$. In
  particular, $\psi([a_1,b_1] \cdot \dots \cdot [a_N,b_N])$ lies in
  the normal subgroup of~$F(S)$ generated by~$r$; hence, there
  exist~$n \in \N$, $t_1,\dots, t_n \in F(S)$, and
  $\varepsilon_1,\dots, \varepsilon_n \in \{-1,1\}$ with
  \[ \cl_S(t_1 \cdot r^{\varepsilon_1} \cdot t_1^{-1} \cdot \dots\cdot
  t_n \cdot r^{\varepsilon_n} \cdot t_n^{-1}) \leq
      \cl_S\bigl( \psi([a_1,b_1] \cdot \dots \cdot [b_1,b_N])\bigr) \leq N.
  \]
  This shows that
  \[ 4 \cdot \bigl( \cl_S(t_1 \cdot r^{\varepsilon_1} \cdot t_1^{-1} \cdot \dots\cdot
  t_n \cdot r^{\varepsilon_n} \cdot t_n^{-1}) -1 \bigr)
  \leq 4 \cdot (N-1)
  = - 2 \cdot \chi(\Sigma).
  \]
  Furthermore, the same arguments as above imply that $n(f,\Sigma) =
  \varepsilon_1 + \dots + \varepsilon_n$; in particular,
  $\varepsilon_1 + \dots + \varepsilon_n \neq 0$ and~$n
  >0$. Therefore, we obtain
  \[ c(r) \leq \frac{-2 \cdot \chi(\Sigma)}{\bigl|n(f,\Sigma)\bigr|}.
  \]
  By Proposition~\ref{prop:svsurface}, taking the infimum over all
  $l^1$-admissible maps shows that
  \[ c(r) \leq \|G_r\|,
  \]
  as claimed. 
\end{proof}

\begin{prop}[weak lower bound]\label{prop:weaklower}
  In the situation of Setup~\ref{setup:onerel}, we have
  \[ \inf_{n \in \N_{>0}} \frac{\cl_S(n \cdot r) - 1}{n}
     \geq \scl_S(r) - \frac12.
  \]
\end{prop}
\begin{proof}
  Let $n \in \N_{>0}$, let $t_1,\dots, t_n \in F(S)$, and let
  \[ N:= \cl_S (t_1 \cdot r \cdot t_1^{-1} \cdot \dots \cdot t_n \cdot r \cdot t_n^{-1});
  \]
  then~$N>0$ and we can geometrically implement this by an
  $\scl$-admissible map~$(f,\Sigma) \in \Sigma_\partial(r)$ with
  \[ n(f,\Sigma) = n \qand
  \chi(\Sigma) = 2 - 2 \cdot N - n.
  \]
  Using the description of~$\scl$ in terms of surfaces
  (Proposition~\ref{prop:stable commutator length via surfaces}),
  we obtain
  \begin{align*}
    \scl_S r
    & \leq \frac{-\chi(\Sigma)}{2 \cdot n(f,\Sigma)}
    %\\
    %&
    = \frac{\cl_S (t_1 \cdot r \cdot t_1^{-1} \cdot \dots \cdot t_n \cdot r \cdot t_n^{-1}) - 1}n + \frac12.
  \end{align*}
  Taking the infimum over all~$n \in \N_{>0}$ and all~$t_1, \dots, t_n\in F(S)$ proves
  the claim.
\end{proof}

%%%%%%%%%%
\subsection{Simplicial volume via quasimorphisms}
\label{subsec:simvol scl via bc}

Stable commutator length in the free group can be computed using
quasimorphisms via Bavard's duality theorem (Theorem~\ref{thm:Bavard}).
We obtain a similar result for the simplicial
volume of one-relator groups:

\begin{prop}[simplicial volume via quasimorphisms]\label{prop:svviaqm}
  Let $S$ be a set and $r \in F(S)'\setminus\{e\}$. Then
  $$
  \| G_r \| = \sup_{\phi \in Q(r)} \frac{\phi(r)}{D(\phi)},
  $$
  where $Q(r)$ is the space of all quasimorphisms $\phi \col F(S) \to
  \R$ satisfying that for all $g,h \in F(S)$ we have $\phi(g \cdot h r
  h^{-1}) = \phi(g) + \phi(h r h^{-1})$.
\end{prop}

\begin{proof}
  In view of the duality principle (Proposition~\ref{prop:duality
    simvol bc}), it suffices to look at~$H^2_b(G_r;\R)$ to
  compute~$\|G_r\|$.  Let $\omega \in C^2_b(G_r;\R)$ be a bounded
  (bar) cocycle on~$G_r$ that is dual to the fundamental
  class~$\alpha_{r,\R} \in H_2(G_r;\R)$, i.e., such that $\langle
  [\omega], \alpha_{r,\R}\rangle = \|G_r\|$.  We may assume that
  $\omega$ is alternating and thus that $\omega(g,e) = 0$ for
  all~$g\in G_r$.

  Let $\widetilde \omega \in C^2_b(F(S);\R)$ denote the pullback
  of~$\omega$ via the canonical projection~$F(S) \longrightarrow
  G_r$. Then, because of $H^2(F(S);\R) \cong 0$, there exists a
  quasimorphism~$\phi \colon F(S) \longrightarrow \R$ on~$F(S)$ such
  that $\delta^1 \phi = \widetilde \omega$ and $D(\phi) = \|\widetilde
  \omega\|_\infty = \|\omega\|_\infty$.

  For all~$h \in F(S)$, the conjugate~$h\cdot r \cdot h^{-1}$
  represents the neutral element in~$G_r$. Therefore, using that
  $\omega$ is alternating, we see that
  \[   \delta^1\phi(g, h \cdot r \cdot h^{-1})
  = \widetilde \omega(g,h \cdot r \cdot h^{-1})
  = \omega([g], e)
  = 0
  \]
  for all~$g,h \in F(S)$. Therefore,
  $ \phi(g) + \phi(h \cdot r \cdot h^{-1}) = \phi(g \cdot h \cdot r \cdot h^{-1}) 
  $
  for all~$g,h \in F(S)$, as claimed. 
\end{proof}

Moreover, we have the following relationship between $\scl$-extremal
and $l^1$-extremal quasimorphisms:
\begin{prop}
  Let $S$ be a set, let $r \in F(S)'$, and for~$N \in \N$ 
  let $\phi_N$ be an $l^1$-extremal quasimorphism to~$r^N$
  (i.e., $\|G_{r^N}\| = \phi_N(r^N)$) with defect~$1$.
  Further, let $\Omega$ be a non-principal ultrafilter on $\N$ and let
  \begin{align*}
    \psi \colon F(S) & \longrightarrow \R \\
    g & \longmapsto \frac14 \cdot \lim_{N \in \Omega} \frac{\phi_N(g)}N,
  \end{align*}
  where $\lim_{N \in \Omega}$ denotes the ultralimit along~$\Omega$.
  Then $\overline{\psi}$, the homogenisation of~$\psi$, is an
  $\scl$-extremal quasimorphism for~$r$, i.e., $\scl_S(r) =
  \overline\psi(r)/D(\overline\psi)$.
\end{prop}

\begin{proof}
Using the properties of ultralimits we may estimate for all~$g,h \in
F(S)$,
\[
1 \geq \lim_{N\in\Omega} \bigl|\phi_N(g) + \phi_N(h) - \phi_N(g\cdot h)\bigr|
  =  \bigl| \psi(g) + \psi(h) - \psi(g\cdot h) \bigr|
\]
and hence $\psi$ is a quasimorphism with defect~$D(\psi) \leq 1$.
Therefore, the homogenisation~$\overline \psi \colon r \longmapsto
\lim_{N \rightarrow \infty} \psi(r^N)/N$ satisfies~$D(\overline{\psi})
\leq 2$ and (where ``$\oplus C$'' means up to error at most~$\pm C$)
\begin{align*}
  \overline{\psi}(r)
  & = \lim_{N \rightarrow \infty} \lim_{K \in \Omega} \frac{\phi_K(r^{N \cdot K})}{N \cdot K}
  & \text{(definition of $\psi$ and $\overline \psi$)}
  \\
  & = \lim_{N \rightarrow \infty} \lim_{K \in \Omega} \frac{N \cdot \phi_K(r^K) \oplus N \cdot 1}{N \cdot K}
  & \text{($\phi_K \in Q(F(S))$ and~$D(\varphi_K) = 1$)}
  \\
  & = \lim_{K \in \Omega} \frac{\| G_{r^K} \|}{K}
  & \text{(by $l^1$-extremality)}
  \\
  & = 4 \cdot \scl_S(r).
  & \text{(Theorem~\ref{thm:powers})}
\end{align*}
From Bavard duality (Theorem~\ref{thm:Bavard}), we obtain
\[
\scl_S(r)
\geq \frac{\overline{\psi}(r)}{2 \cdot D(\overline{\psi})}
\geq \frac{4 \cdot \scl_S(r)}{4}
\]
and hence $\overline{\psi}$ is $\scl$-extremal with defect
$D(\overline{\psi}) = 2$.
\end{proof}

%%%%%%%%%%%%%%%%%%%%%%%%%%%%%%%%%%%%%%%%%%%%%%%%%%%%%%%%%%%%%%%%%%
\section{Van Kampen diagrams on surfaces}\label{sec:vKdiag}

We recall van Kampen diagrams on surfaces, which we will use to encode
the $l^1$-admissible maps of Proposition \ref{prop:svsurface}.  This
allows us to use combinatorial methods to estimate and sometimes
compute the simplicial volume of one-relator groups.  The main result
of this section is the estimate for powers of elements; see
Section~\ref{subsec:simvol for proper powers}.

Parts of this section are an adaptation of corresponding work
on~$\scl$~\cite[Section~4]{Heuer-scl-rp-groups}.  We will estimate the
Euler characteristic of van Kampen diagrams by defining a
combinatorial curvature~$\kappa(D)$ for the disks~$D$ of a van Kampen
diagram in Section \ref{subsec:comb gauss bonnet}. For the theorem on
powers (Theorem \ref{thm:powers}), we will then estimate~$\kappa(D)$,
using \emph{branch vertices} in Section \ref{subsec:comb gauss
  bonnet}. In Section~\ref{subsec:simvol for proper powers}, we will
prove the theorem estimating the simplicial volume of one relator groups
where the relation is a proper power.

%%%%%%%%%%%%%%%%%%%%
\subsection{$l^1$-Admissible surfaces via van Kampen diagrams}

Van~Kampen diagrams on surfaces have been introduced by Olshanskii to
study homomorphisms from surface groups to a group with a given
presentation~\cite{OL,CSS}.

\begin{defn}[van~Kampen diagram]
  Let $r \in F(S)'\setminus\{e\}$ and let $P_r$ be the presentation
  complex of~$G_r = \genrel S r$ as in Setup~\ref{setup:onerel};
  furthermore, let $\Sigma$ be an oriented closed surface.  A
  \emph{van~Kampen diagram~$\Dcl$ for the presentation~$r$
    on~$\Sigma$} is a decomposition of~$\Sigma$ into finitely many
  polygons, also called \emph{disks}, where the edges are labelled by
  words over~$S^{\pm}$ such that the boundary of each disk is labelled
  counterclockwise (i.e., orientation-preservingly with respect to the
  orientation induced from~$\Sigma$) in a reduced way by~$r^+$
  or~$r^{-}$.  Moreover, the labels of edges of adjacent disks are
  required to be compatible, i.e., if an edge is adjacent to two
  disks, then the label of one edge is $w \in F(S)$ and the label of
  the other one is~$w^{-1}$.  The underlying surface of~$\Dcl$ is
  denoted by~$\Sigma_\Dcl$.  For a disk $D$ in a van~Kampen diagram
  labelled by $r^\epsilon$ we call $\epsilon$ the sign of $D$.  The
  \emph{total degree} of the van~Kampen diagram is defined as $\sum_{D
    \in \Dcl} n(D)$ where the sum runs over all disks of $\Dcl$.
  
  We write~$\Delta(r)$ for the ``set'' of all van~Kampen diagrams for~$r$.  
\end{defn}

Every van Kampen diagram~$\Dcl$ for~$r$ induces a continuous
map~$f_{\Dcl} \col \Sigma_\Dcl \to P_r$ to the presentation complex of
$G_r$ by mapping the labelled edges to the edges in the $1$-skeleton
of~$P_r$ and mapping the disks to the $2$-cell of~$P_r$.  Every
such map is $l^1$-admissible in the sense of Definition
\ref{def:admissiblel1}; the degree of this map is the difference of
the number of positive and negative disks.  Conversely, we may
replace every $l^1$-admissible map by a map induced by a van Kampen
diagram; thus van Kampen diagrams may be used to compute~$\| G_r \|$:

\begin{prop}[simplicial volume via van~Kampen diagrams] \label{prop:admissible maps and van Kampen diagrams}
  In the situation of Setup~\ref{setup:onerel}, if $r$ is cyclically
  reduced, we have
  \[ \| G_r \| = \inf_{\Dcl \in \Delta(r)}
                 \frac{-2 \cdot \chi^-(\Sigma_\Dcl)}%
                      {\bigl|n(f_{\Dcl},\Sigma_\Dcl)\bigr|}.
  \]
\end{prop}

  \begin{figure}
    \begin{center}
      \def\sgvertx#1{%
        \fill #1 circle (0.05);}
      \begin{tikzpicture}[x=1cm,y=1cm]
        % radial sections
        \begin{scope}[black!50]
          \draw[->] (1.5,0) -- (2.9,0);
          \draw[->] (60:1.5) -- (60:2.9);
          \draw (2.25,0) node[anchor=north] {$c_1$};
          \draw (60:2.25) node[anchor=south east] {$c_2$};
        \end{scope}
        % filling in the inner polygon
        \begin{scope}[black!50]
          \draw[->] (0,0) -- (1.4,0);
          \draw[->] (0,0) -- (60:1.4);
          \draw (0.75,0) node[anchor=north] {$e$};
          \draw (60:0.75) node[anchor=south east] {$e$};
        \end{scope}       
        % outer polygon
        \draw (0,0) circle (3);
        \gvertx{(3,0)}
        \gvertx{(10:3)}
        \gvertx{(20:3)}
        \gvertx{(30:3)}
        \gvertx{(40:3)}
        \gvertx{(60:3)}
        \draw (5:3.3) node {$a_1$};
        \draw (15:3.3) node {$b_1$};
        \draw (25:3.4) node {$a_1^{-1}$};
        \draw (35:3.4) node {$b_1^{-1}$};
        
        % inner polygon
        \draw (0,0) circle (1.5);
        \gvertx{(1.5,0)}
        \gvertx{(60:1.5)}
        \sgvertx{(20:1.5)}
        \sgvertx{(40:1.5)}
        \draw (10:1.8) node {$t_1$};
        \draw (30:1.8) node {$r$};
        \draw (50:1.9) node {$t_1^{-1}$};

        % the new loop
        \draw[->,blue] (4:1.6) -- (2:2.9) arc (2:58:2.9) -- (56:1.6);
        \draw[blue] (30:2.6) node {$w_1$};
        % the centre
        \gvertx{(0,0)}
      \end{tikzpicture}
    \end{center}
    
    \caption{From Equation~\ref{eq:doublepolygon} to a van~Kampen diagram}\label{fig:doublepolygon}
  \end{figure}
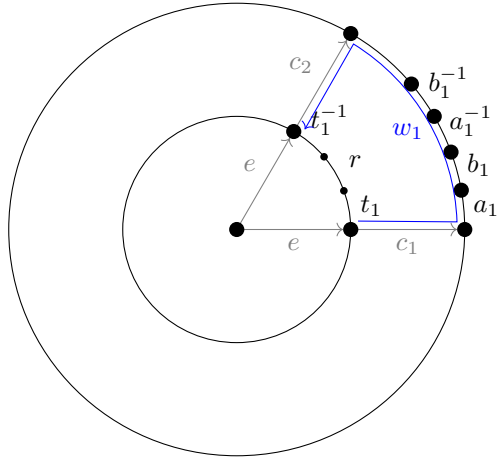

Here, $\chi^-$ denotes the Euler characteristic that ignores spherical
components, i.e., for a surface~$\Sigma = \bigsqcup_{i}^n \Sigma_i$
with connected components~$\Sigma_i$ we have that
\[
\chi^-(\Sigma) := \sum_{i=1}^n \min \bigl( 0, \chi(\Sigma_i)\bigr).
\]

\begin{proof}
  Because van Kampen diagrams induce $l^1$-admissible maps, the
  inequality~``$\leq$'' holds.  For the converse estimate, we use the
  description of~$\|G_r\|$ from Corollary~\ref{cor:algsimvol}.  Let $n
  \in \N_{>0}$, let $\varepsilon_1, \dots, \varepsilon_n \in
  \{1,-1\}$, let $t_1,\dots, t_n \in F(S)$, and let
  \[ N := \cl_S(t_1\cdot r^{\varepsilon_1} \cdot t_1 \cdots t_n \cdot r^{\varepsilon_n} \cdot t_n^{-1}) > 0.
  \]
  It then suffices to construct a van~Kampen diagram for~$r$ with
  $n$~disks with the signs~$\varepsilon_1,\dots, \varepsilon_n$ on
  an oriented closed connected surface of genus~$N$ (the degree of the
  associated map will be~$\varepsilon_1 + \dots + \varepsilon_n$ and
  the Euler characteristic of the surface will be~$2- 2 \cdot N$).

  By definition of~$N$, there exist~$a_1, \dots, a_N, b_1, \dots, b_N \in F(S)$ with
  \begin{align} %c :=
    \label{eq:doublepolygon}
          t_1\cdot r^{\varepsilon_1} \cdot t_n \cdots t_n \cdot r^{\varepsilon_n} \cdot t_n^{-1}
        = [a_1, b_1] \cdots [a_N, b_N].
  \end{align}
  We now consider a $4N$-gon, whose edges are labelled by~$a_1, \dots,
  b_N$; inside, we put an $n$-gon, whose edges are labelled by~$t_1
  \cdot r^{\varepsilon_1}\cdot t_1^{-1} ,\dots, t_n \cdot
  r^{\varepsilon_n} \cdot t_n^{-1}$
  (Figure~\ref{fig:doublepolygon}). Because of
  Equation~\eqref{eq:doublepolygon}, the corresponding annulus admits
  a continuous map~$f$ to~$\bigvee_S S^1$ (where the circles are
  labelled by the elements of~$S$) that is compatible with the labels
  of the edges.  We now connect the vertices of the inner disks
  radially (and without crossings) with vertices of the outer disk
  (Figure~\ref{fig:doublepolygon}); we label these radial
  sectors~$c_1, \dots, c_n$ by the elements of~$F(S)$ represented by
  the corresponding loop in~$\bigvee_S S^1$ via~$f$. For~$j \in
  \{1,\dots, n\}$, let $w_j \in F(S)$ be the element obtained by
  following~$c_j$, then walking on the outer disk until the
  endpoint of~$c_{j+1}$, and then following the inverse~$\bar c_{j+1}$
  of~$c_{j+1}$.  By construction, $w_j$ is conjugate
  to~$r^{\varepsilon_j}$ in~$F(S)$.

  As next step, we fill in the inner disk by $n$~radial
  sectors~$d_1,\dots, d_n$, all labelled by~$e$.  We now subdivide all
  edges according to reduced representations over~$S^{\pm}$ (or~$e$)
  of their labels. In this way, we obtain an oriented closed connected
  surface~$\Sigma$ of genus~$N$ that is decomposed into~$n$ compatibly
  edge-labelled disks, each of which is labelled by a conjugate
  of~$r^\pm$.

  It remains to reduce the words labelling the boundaries of the disks.
  We first contract all edges labelled by~$e$ to points; this leads to
  a homeomorphic surface (no pathologies can occur because $N > 0$).
  If the label of the boundary of a disk is \emph{not} reduced, we
  may reduce it by gluing the corresponding edges (Figure~\ref{fig:paste});
  this reduces the number of unreduced positions in the label of this
  disk and leaves all other labels unchanged. Therefore, inductively,
  we obtain a decomposition of~$\Sigma$ into $n$~disks with cyclically
  reduced boundary labels that are conjugate to~$r^{\varepsilon_1}, \dots, r^{\varepsilon_n}$.
  Because $r$ is cyclically reduced, this means that each disk is labelled
  by (a cyclic shift of)~$r^\pm$~\cite[Theorem~IV.1.4]{lyndonschupp}. Therefore,
  we obtain the desired van~Kampen diagram for~$r$.
\end{proof}

\begin{figure} 
\begin{center}
\includegraphics[scale=0.5]{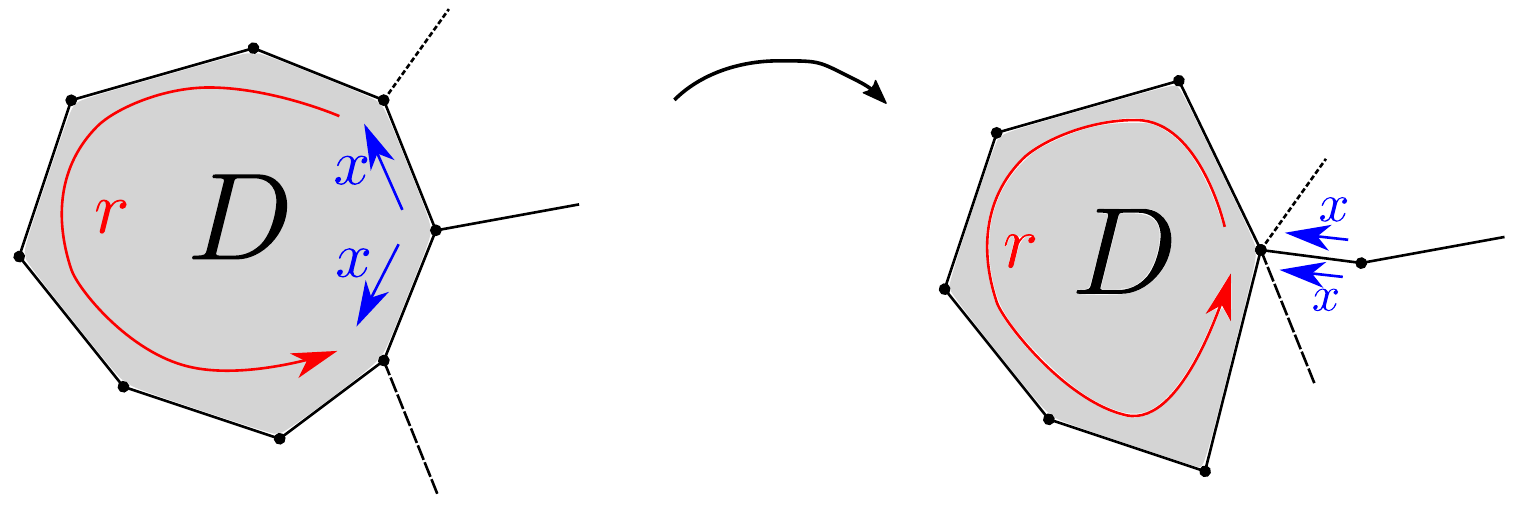}
\caption{If the boundary word of $D$ has backtracking we may glue up
  the backtracking and replace it by a disk with shorter boundary
  word.} \label{fig:paste}
\end{center}
\end{figure}

%%%%%%%%%%%%%%%%%%%%%
\subsection{Combinatorial Gau\ss-Bonnet} \label{subsec:comb gauss bonnet}

Let $\Dcl$ be a van~Kampen diagram and let $D$ be a disk of~$\Dcl$
(we will also abbreviate this by writing~``$D \in \Dcl$''). 
Recall that every van~Kampen diagram has an associated surface $\Sigma_{\Dcl}$ such that the disks in $\Dcl$ decompose $\Sigma_{\Dcl}$ into finitely many polygons glued together along their edges. A \emph{vertex} $v$ of $\Dcl$ is a vertex of those disks and $\deg(v)$ denotes the \emph{degree} of~$v$,
i.e., the number of edges adjacent to~$v$ in~$\Sigma_{\Dcl}$.  Morever, we
write $V_D$ for the set of all vertices and $E_D$ be the set of all
edges of~$D$.

\begin{defn}[curvature in a van~Kampen diagram]\label{def:curvature}
  Let $\Dcl$ be a van~Kampen diagram. Then the \emph{curvature} of disks
  of~$\Dcl$ is define by
  $$
  \kappa(D) := \sum_{v \in V_D} \Bigl(\frac{1}{\deg(v)} - \frac{1}{2} \Bigr) + 1.
  $$
\end{defn}

\begin{prop}[combinatorial Gau\ss-Bonnet] \label{prop:combinatorial gauss bonnet}
Let $\Dcl$ be a van Kampen diagram on a surface $\Sigma_\Dcl$. Then
$$
\chi(\Sigma_\Dcl) = \sum_{D \in \Dcl} \kappa(D).
$$
\end{prop}

\begin{proof}
  Every vertex in $\Sigma_\Dcl$ is adjacent to $\deg(v)$ many
  disks. Thus the total number of vertices in~$\Sigma_\Dcl$
  equals~$\sum_{D \in \Dcl} \sum_{v \in V_D}
  \frac{1}{\deg(v)}$. Similarly, $\sum_{D \in \Dcl} \sum_{e \in E_D}
  \frac{1}{2}$ is the total number of edges as every edge is counted
  twice in the two adjacent disks; and $\sum_{D \in \Dcl} 1$ is the
  total number of disks.  Hence,
  \[
  \sum_{D \in \Dcl} \kappa(P)
  = \# \text{vertices}
  - \# \text{edges}
  + \# \text{faces}
  = \chi(\Sigma_{\Dcl}). \qedhere
  \]
\end{proof}

If $D$ is a disk in a van~Kampen diagram, we may
estimate~$\kappa(D)$ in terms of the number of vertices of degree at
least~$3$, so-called \emph{branch vertices}.

\begin{prop} \label{prop:branch vertices}
  Let $\Dcl$ be a van~Kampen diagram, let $D$ be a disk of~$\Dcl$,
  and let $\beta(D)$ be the number of branch vertices of~$D$. Then
  $$
  \kappa(D) \leq \frac{6-\beta(D)}{6}.
  $$
\end{prop}

\begin{proof}
Every vertex in the disk~$D$ has degree at least~$2$. Thus we compute
\begin{align*}
\kappa(D) &= \sum_{v \in V_D} \Bigl(\frac{1}{\deg(v)} - \frac{1}{2} \Bigr) + 1 \\
&\leq   \sum_{\deg(v) \geq 3} \Bigl(\frac{1}{3} - \frac{1}{2} \Bigr)  + 1  \\
&\leq  1 - \frac{\beta(D)}{6} = \frac{6-\beta(D)}{6}.
\qedhere
\end{align*}
\end{proof}

\subsection{Strong bounds from hyperbolicity}\label{subsec:simvol for proper powers}

It is generally not known which one-relator groups are
hyperbolic~\cite{cashen2020short, louder2018negative}.  However, there
are two types of elements in $F(S)$ for which hyperbolicity is
well-known: proper powers and small cancellation elements. In both
cases, we obtain strong lower bounds for the simplicial volume in
terms of stable commutator length.  The key insight is the following
lemma:

\begin{lemma} \label{lemma:high beta and simvol}
In the situation of Setup~\ref{setup:onerel}, suppose that there is
an~$N \geq 7$ such that the infimum in
Proposition~\ref{prop:admissible maps and van Kampen diagrams} is
achieved as the infimum over van Kampen diagrams~$\Dcl$ such that
$\beta(D) \geq N$ for every~$D \in \Dcl$.  Then
\[
\Bigl(1-\frac{6}{N} \Bigr) \cdot 4 \cdot \scl_S(r) \leq \| G_r \|.
\]
\end{lemma}

\begin{proof}
Let $\epsilon > 0$.  Choose a van Kampen diagram~$\Dcl$ on a
surface~$\Sigma_\Dcl$ such that $\beta(D) \geq N$ for every $D \in
\Dcl$ and such that
\[
\| G_{r} \| \geq \frac{- 2 \cdot \chi^-(\Sigma_{\Dcl})}{n(f_{\Dcl}, \Sigma_{\Dcl})} - \epsilon.
\]
By removing spherical components we may assume that
$\chi^-(\Sigma_{\Dcl}) = \chi(\Sigma_{\Dcl})$.  Let $m^+$ be the
number of positive disks and let $m^-$ be the number of negative disks
of $\Dcl$. Then the degree is $n(f_{\Dcl},\Sigma_{\Dcl}) = m^+ - m^-$
and the total number of disks is $m^+ + m^-$.  Using the combinatorial
Gau\ss-Bonnet Theorem (Proposition~\ref{prop:combinatorial gauss
  bonnet}) and Proposition~\ref{prop:branch vertices}, we see that
$$
  \chi^-(\Sigma_{\Dcl})
= \chi(\Sigma_{\Dcl})
= \sum_{D \in \Dcl} \kappa(D) \leq  \frac{6-N}{6} \cdot (m^+ + m^-)
$$
and hence
$$
  \| G_{r} \| + \epsilon \geq   \frac{-2 \cdot \chi^-(\Sigma_{\Dcl})}{n(\Sigma_{\Dcl}, f_{\Dcl})}
  \geq  2 \cdot \frac{N-6}{6} \cdot \frac{m^+ + m^-}{m^+ - m^-}.
$$
We conclude that 
$$
  \frac{1}{2} \cdot \frac{6}{N-6}\cdot \bigl(\| G_{r} \|+\epsilon\bigr) \geq 
  \frac{m^+ + m^-}{m^+ - m^-}.
$$

Let $\Sigma_{\partial}$ be the surface obtained by removing the $(m^+
+ m^-)$ disks of $\Sigma_{\Dcl}$. Then $\Sigma_{\partial}$ contracts
to the $1$-skeleton of~$P_r$ via~$f_{\partial}$ and every boundary
word of~$\Sigma_{\partial}$ maps to a word labelled by~$r^\pm$. Thus
$(f_{\partial}, \Sigma_{\partial})$ is $\scl$-admissible for~$r$;
see Definition \ref{defn:scl admissible maps}. We see that
$$
  \chi^-(\Sigma_{\partial}) = \chi^-(\Sigma_{\Dcl})-(m^+ + m^-)
$$  
and observe that the $\scl$-degree of $(f_{\partial},
\Sigma_{\partial})$ is $(m^+ - m^-)$.

This leads to the estimate
\begin{eqnarray*}
  \scl_S(r) &\leq &\frac{- \chi^-(\Sigma_\partial)}{2 \cdot (m^+ - m^-)} \\
  %& = & \frac{-\chi^-(\Sigma_{\Dcl})+m^+  + m^-)}{2 \cdot (m^+ - m^- )} \\
  & = & \frac{-\chi^-(\Sigma_{\Dcl})}{2 \cdot (m^+ - m^-)} + \frac{m^+ + m^-}{2 \cdot (m^+ - m^-) } \\
  &\leq & 
  \frac{1}{4} \cdot \bigl( \| G_{r} \| + \epsilon \bigr) \cdot \Bigl( 1 + \frac{6}{N-6} \Bigr).
\end{eqnarray*}
As this inequality holds for every~$\epsilon$ we conclude that
\begin{eqnarray*}
  \scl_S(r) &\leq & 
  \frac{1}{4} \cdot \| G_{r} \| \cdot \Bigl( 1 +  \frac{6}{N-6} \Bigr)
\end{eqnarray*}
and by rearranging terms that
\[
\Bigl(1-\frac{6}{N} \Bigr) \cdot 4 \cdot \scl_S(r) \leq \| G_r \|.
\qedhere
\]
\end{proof}

We apply Lemma \ref{lemma:high beta and simvol} to the case of small cancellation elements. Recall that two elements~$w, v \in F(S)$ are said to \emph{overlap} in the word~$x$, if $x$ is a prefix of both $w$ and $v$, i.e.\ if we may write $w = x \cdot w'$ and $v = x \cdot v'$ as reduced words, for an adequate choice of~$w', v'$.
An element~$r \in F(S)$ is said to satisfy small cancellation condition $C'(1/N)$, if whenever it overlaps in $x$ with a cyclic conjugate of $r$ or $r^{-1}$ that is not equal to~$r$, then $|x| \leq \frac{1}{N} \cdot |r|$. 

\begin{thm}[small cancellation elements] \label{thm:small cancellation}
  In the situation of Setup~\ref{setup:onerel}, if the relator~$r \in
  F(S)'\setminus\{e\}$ satisfies the small cancellation
  condition~$C'(1/N)$ for some~$N \geq 7$, then
  $$
\Bigl( 1 - \frac{6}{N} \Bigr) \cdot 4 \cdot \scl_S(r) \leq \| G_r \| < 4 \cdot \scl_S(r).
  $$.
\end{thm}

\begin{proof}
  If $r$ satisfies the small cancellation condition~$C'(1/N)$, then
  for every van Kampen diagram we have that all disks~$D$
  satisfy~$\beta(D) \geq N$. Thus, the first inequality follows from
  Lemma~\ref{lemma:high beta and simvol}. The second inequality holds
  generally; see Corollary \ref{cor:weakupper}.
\end{proof}

\begin{rmk}
We note that an element $r$ that satisfies small cancellation condition~$C'(1/N)$ satisfies $\scl_S(r) \geq \frac{N-6}{12}$. This can be seen by considering the branch edges, similarly to Proposition \ref{prop:branch vertices} and computing the Euler characteristic for the corresponding surface with boundary. 
Thus, using Lemma \ref{lemma:high beta and simvol}, we see that
$$
\| G_r \| \geq \Bigl(1-\frac{6}{N} \Bigr) \cdot 4 \cdot \frac{N-6}{12} = \frac{N}{3} - 1
$$
for all $r$ which satisfy small cancellation condition $C'(1/N)$ with $N > 6$.
\end{rmk}

On the other hand, we see that for sufficiently large powers, we get a
strong connection between stable commutator length and the simplicial
volume of one-relator groups.

\begin{thm}[proper powers]\label{thm:powers}
  In the situation of Setup~\ref{setup:onerel}, we have for all~$N > 6$:
  $$
  \Bigl( 1 - \frac{6}{N} \Bigr) \cdot 4 \cdot \scl(r^N) \leq \| G_{r^N} \| < 4 \cdot \scl_S(r^N).
  $$
 In particular, we obtain
  $$
  \lim_{N \to \infty} \frac{\| G_{r^N} \|}{N} = 4 \cdot \scl_S(r).
  $$ 
\end{thm}

\begin{proof}
  The second inequality follows from the weak upper bound for
  simplicial volume (Corollary \ref{cor:weakupper}). To see the other
  inequality we will use Lemma \ref{lemma:high beta and simvol} and
  show that $\| G_{r^N} \|$ may be approximated by van~Kampen diagrams
  $\Dcl$ that satisfy $\beta(D) \geq N$ for all $D \in \Dcl$.

  Let $r = \xtt_0 \cdots \xtt_{n-1}$ with $\xtt_i \in S$ be the
  reduced word representing~$r$; we may assume that $r$ is cyclically
  reduced and not a proper power. 

\begin{claim} \label{claim:popwer wordlength less n}
  Let $\Dcl$ be a van Kampen diagram on a surface $\Sigma_{\Dcl}$ over
  $r^N$ such that for every van Kampen diagram $\Dcl'$ on a surface
  $\Sigma_{\Dcl'}$ over $r^N$ with fewer disks than~$\Dcl$ we have
  that
  \begin{eqnarray} \label{equ:minimal tiling}
    \frac{-2\cdot \chi^-(\Sigma_{\Dcl})}{n(f_{\Dcl},\Sigma_{\Dcl})} <  \frac{-2\cdot\chi^-(\Sigma_{\Dcl'})}{n(f_{\Dcl'},\Sigma_{\Dcl'})}.
  \end{eqnarray}
  Let $D \in \Dcl$ be a disk and let $e \subset \partial D$ be a
  connected subpath of the boundary of~$D$ such that $e$ has no branch
  vertices in the interior.

  Then the label of~$e$ has word length strictly less than~$|r|=n$. 
\end{claim}

\begin{proof}
  Without loss of generality we assume that $D$ is positive, i.e., its
  boundary is labelled by~$r^N$.  Thus the label $w \in F(S)$ of $e$
  is a reduced subword of~$r^N$ (cyclically written). \emph{Assume}
  for a contradiction that $|w| \geq n$. Then, by cyclically
  relabelling~$r$ we may assume that $w = \xtt_0 \cdots \xtt_{n-1}
  \tilde{r}$.

  Because $e$ has no branch vertices in its interior, there is a
  disk~$D' \in \Dcl$ that is adjacent to~$e$; let $e'$ be the
  subpath of the boundary of~$D'$ that corresponds to~$e$ in~$D$. Then
  the label of~$e'$ is~$w^{-1}$.  We consider two different cases:
  \begin{itemize}
  \item The disk~$D'$ is positive. Then $w^{-1}$ is a subword of
    $r^N$ (cyclically written) as $e'$ is an edge of $D'$ and the
    boundary of $D'$ is labelled by the word~$r^N$.  Suppose that the
    word $w^{-1}$ ends in $\xtt_i$. Then we see that $\xtt_0^{-1} =
    \xtt_i$.  Similarly, we see that $\xtt_1^{-1} = \xtt_{i-1}$ and
    $\xtt_k^{-1} = \xtt_{i-k}$ for every $k < i$. If $i$ is even, then
    this implies that $\xtt_{i/2}^{-1} = \xtt_{i/2}$, which is a
    contradiction; if $i$ is odd, this implies that
    $\xtt_{i/2-1/2}^{-1} = \xtt_{i/2+1/2}$, which contradicts that $r$
    is a reduced word.
  \item The disk~$D'$ is negative.  In this case $w^{-1}$ is a
    subword of $r^{-N}$.  By adding degree~$2$ vertices to the
    disks of the van~Kampen diagram~$\Dcl$, we may assume that
    every edge is labelled by a single letter in~$S^\pm$.  Suppose
    that the boundary of $D$ is $e \cdot f$ and that the boundary of
    $D'$ is $f' \cdot e'$. Here, $a \cdot b$ denotes the concatenation
    of two paths $a$ and $b$.

    Then $e$ may be written as $e = e_{0} \cdot \cdots e_{n-1} \cdot
    \tilde{e}$ where $e_{i}$ is labelled by $\xtt_i$ for all~$i \in \{0,\dots, n-1\}$
    and $\tilde{e}$ is labelled by $\tilde{r}$.
    Similarly, $e'$ may be written as $e'= \tilde{e}' \cdot e'_{n-1}
    \cdots e'_{0}$, where $e_{i}'$ is labelled by $\xtt_i^{-1}$ for
    all~$i \in \{0, \ldots, n-1 \}$.

    \begin{figure} 
      \centering
      \includegraphics[scale=0.74]{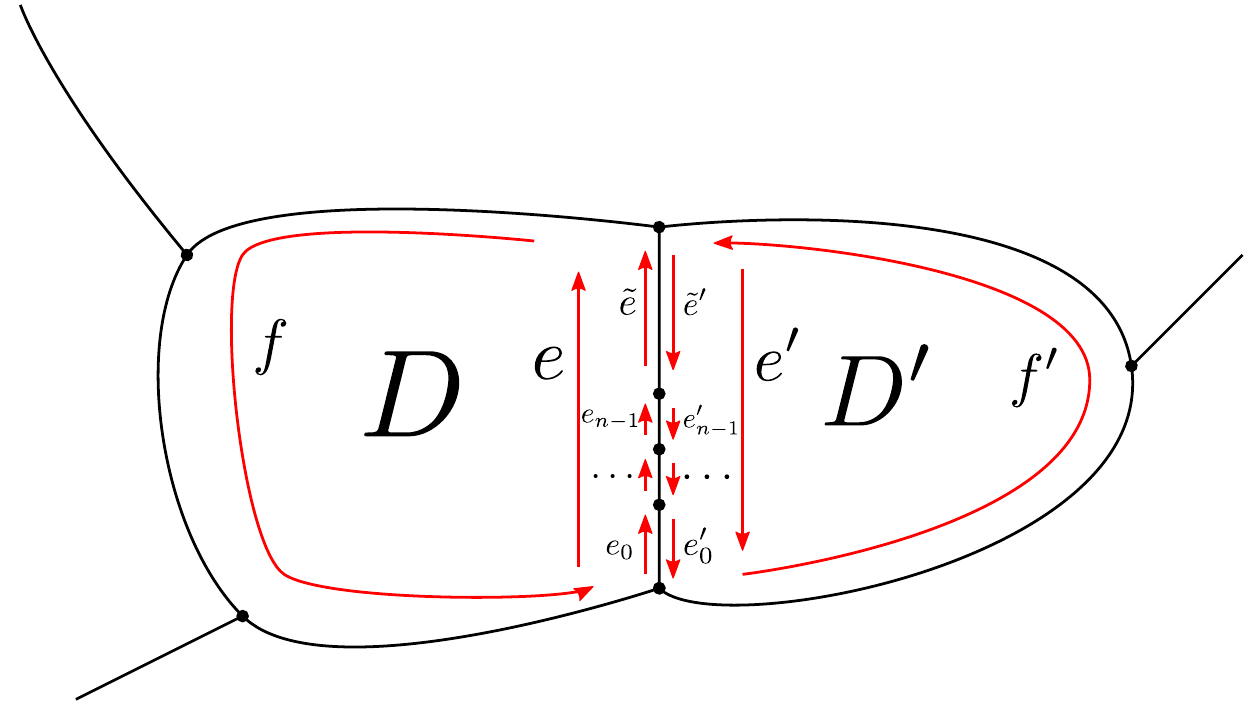}
      \caption{The disks $D$ and $D'$ share the subpaths $e$ and $e'$ in $\Sigma_{\Dcl}$.}
      \label{fig:inverse_disks}
    \end{figure}

    The boundary labels of both $D$ and $D'$ are $n$-periodic, 
    i.e., after $n$ segments the labels repeat.  Thus the first edge
    of $\tilde{e}$ has to be labelled by $\xtt_0$ and the last edge of
    $\tilde{e}'$ has to be labelled by $\xtt_0^{-1}$. If we continue
    comparing the labels of the edges in this way we see that the
    label for $\tilde{e}$ is inverse to the label for $\tilde{e}'$ and
    that the label for $f$ is inverse to the label for $f'$ (see
    Figure \ref{fig:inverse_disks}).

    Now we may glue both $D$ and $D'$ together along the boundaries as
    in Figure \ref{fig:inverse_disks_glueing}.

    \begin{figure}  
      \centering
      \makebox[0pt]{\includegraphics[scale=1]{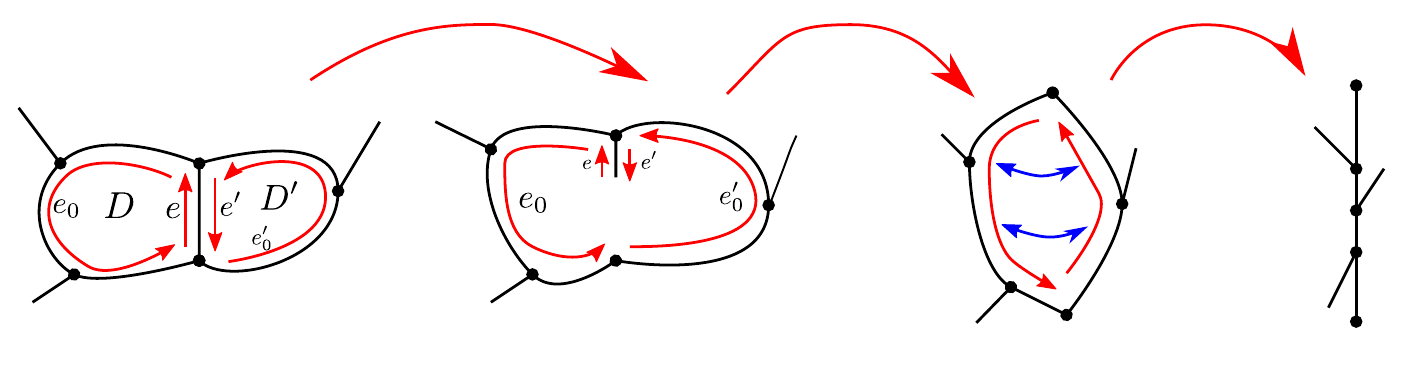}}
      \caption{Gluing up $D$ and $D'$.}
      \label{fig:inverse_disks_glueing}
    \end{figure}

    This procedure does not change the surface $\Sigma_{\Dcl}$ up to
    homotopy equivalence.  The result is a van Kampen diagram on a 
    surface with the same Euler characteristic. The resulting van
    Kampen diagram also has the same degree as $\Dcl$ as the degree of
    both $D$ and $D'$ cancelled. This contradicts the minimality of
    Equation \ref{equ:minimal tiling}.
  \end{itemize}
  In both cases we contradicted that the label had word length at
  least~$|r|$. This proves Claim~\ref{claim:popwer wordlength less n}.
\end{proof}

\begin{claim}\label{claim:conditions for lemma}
  We may approximate $\| G_{r^N} \|$ by a sequence of van~Kampen
  diagrams $\Dcl$ that satisfy that $\beta(D) \geq N$ for all $D \in
  \Dcl$.
\end{claim}

\begin{proof}
  Let $\epsilon > 0$ and let $\Sigma_\Dcl$ be a surface with
  van~Kampen diagram $\Dcl$ with the least number of disks such that
  \[
  \| G_{r^N} \| \geq \frac{-2 \chi^-(\Sigma_\Dcl)}{n(f_{\Dcl}, \Sigma_{\Dcl})} - \epsilon.
  \]
  Let $D \in \Dcl$ be a disk in $\Dcl$.  Claim~\ref{claim:popwer
    wordlength less n} implies that each connected subpath of the
  boundary of~$D$ without branch vertices has length less
  than~$|r|$. Thus there are at least~$|r^N|/|r| = N$ such subpaths
  and branch vertices, i.e., $\beta(D) \geq N$.
\end{proof}

\emph{Conclusion of the proof of Theorem~\ref{thm:powers}.} 
Applying Claim~\ref{claim:conditions for lemma} to Lemma~\ref{lemma:high beta and simvol},
we see that
$$
\Bigl(1-\frac{6}{N} \Bigr) \cdot 4 \cdot \scl_S(r^N) \leq \| G_{r^N} \|.
$$
Using the multiplicativity of stable commutator length we conclude that
\[
\lim_{N \to \infty} \frac{\| G_{r^N} \|}{ N } = 4 \cdot \scl_S(r).
\qedhere
\]
\end{proof}

%%%%%%%%%%%%%%%%%%%%%%%%%%%%%%%%%%%%%%%%%%%%%%%%%%%%%%%%%%%%%%%
\section{Random one-relator groups}\label{sec:random}

In this section, we describe the large scale distribution of~$\| G_r
\|$ for random elements~$r \in F(S)'$. This is an application of
Theorem~\ref{thm:small cancellation} and a result by Calegari and
Walker on the random distribution of stable commutator lengths in free
groups~\cite{rand_rigid}. The aim of this section is to show the
following result.

\begin{thm} \label{thm: random simvol of elts}
  Let $S$ be a finite set, let $\epsilon > 0$ and $C>1$.
  Then, for every random reduced element~$r \in F(S)$
  of even length~$n$, conditioned to lie in the commutator subgroup~$F(S)'$,
  we have 
$$
\biggl|\|G_r\| \cdot \frac{\log(n)}{n}  - \frac{2\cdot \log(2 |S| -1)}{3} \biggr| \leq \epsilon
$$
  with probability~$1-O(n^{-C})$.
\end{thm}

We derive this result by relating the simplicial volume to the stable
commutator length of random elements using the small cancellation
estimate (Theorem~\ref{thm:small cancellation}) and that random
elements are small cancellation.  The result is then a direct
application of the corresponding result for stable commutator length
by Calegari and Walker:

\begin{thm}[\protect{Calegari-Walker \cite[Theorem 4.1]{rand_rigid}}] \label{thm:cal walker rand}
  Let $S$ be a finite set, let $\epsilon > 0$ and $C>1$.
  Then, for every random reduced element~$r \in F(S)$
  of even length~$n$, conditioned to lie in the commutator subgroup~$F(S)'$,
  we have 
$$
\left| \scl_S(r) \cdot \frac{\log(n)}{n}  - \frac{\log(2 |S| -1)}{6} \right| \leq \epsilon
$$
with probability $1-O(n^{-C})$.
\end{thm}

%%%%%%%%%%%%%%%%%%%
\subsection{Random elements of the commutator subgroup}

We recall well-known results about random elements
of the free group and introduce notation.

\begin{setup}\label{setup:random}
  Let $S$ be a finite set. We set~$k := |S|$ and let $F := F(S)$ be
  the free group over $S$. Furthermore:
\begin{itemize}
\item $F_n$ denotes the set of all words of length~$n$.
\item $F'_n$ denotes the set of all words of length~$n$ that lie
  in the commutator subgroup of~$F$. Here, $n$ is supposed to be even.
\item $A^N_n$ denotes the set of all elements (not necessarily cyclically reduced) of length~$n$ that do \emph{not}
  satisfy the small cancellation condition~$C'(1/N)$.
\end{itemize}
\end{setup}

In this situation, $|F_n| = 2k \cdot (2k - 1)^{n-1}$.  We
recall the following Theorem of Sharp, estimating the size of~$|F'_n|$
relative to~$|F_n|$.

\begin{thm}[\protect{Sharp~\cite[Theorem 1]{sharp2001local}\cite[Theorem 2.1]{rand_rigid}}] \label{thm:sharp}
In the situation of Setup~\ref{setup:random}, we have: 
If $n$ is odd then $F'_n$ is empty.
Moreover, there is an explicit constant~$\sigma^k$ (which
depends only on~$k$) such that
$$
\lim_{n \to \infty} \left| \sigma^k \cdot n^{k/2} \cdot \frac{|F'_n|}{|F_n|} - \frac{2}{(2 \pi)^{k/2}}  \right| = 0, 
$$
where the limit is taken over all \emph{even} natural numbers~$n$. 
\end{thm}

We may crudely estimate the exceptional set~$|A^N_n|$ as follows:

\begin{prop} \label{prop: elts sat small canc}
  In the situation of Setup~\ref{setup:random}, we have for all
  natural numbers~$n, N \geq 1$:
$$
|A^N_n| \leq 3 n^3 \cdot (2k)^2 \cdot (2k - 1)^{n - \frac{n}{N}}
$$
\end{prop}

\begin{proof}
Suppose that $w$ does not satisfy the small cancellation condition $C'(1/N)$.
We consider the following cases:

\begin{enumerate}
\item \label{item: inverse} $w \in A^N_n$ overlaps with a cyclic conjugate of $w^{-1}$ in a piece larger than~$n/N$. Then a cyclic conjugate of $w$ may be written as $v = r v_1 r^{-1} v_2$ such that $|r| = \lceil \frac{n}{N} \rceil$ and $|v| = n$ and where there is at most one cancellation either in $v_1$ or in $v_2$, if $w$ was not cyclically reduced.
Here and throughout this section we say that a word \emph{has a cancellation} if there is a subword~$\xtt \cdot \xtt^{-1}$, where $\xtt$ is a letter in the alphabet.

We will estimate the possible choices for $v$.
We have $2k \cdot (2k-1)^{\lceil \frac{n}{N} \rceil - 1}$ choices for~$r$.
If $|v_1| = n_1$ and $|v_2| = n_2$, then there are at most~$n \cdot 2k \cdot (2k-1)^{n_1 + n_2 - 1}$ choices for
$v_1$ and $v_2$: For any letter there are $(2k - 1)$ choices in order to avoid cancellation with the previous letter, apart from one time where we allow the letter to be an inverse of the previous letter. There are at most~$n$ possiblities where such an inverse may occur, and thus we get a total of~$n \cdot 2k \cdot (2k-1)^{n_1 + n_2 - 1}$.

Thus we estimate that in this case there are a total of
$$
n \cdot (2k)^2 \cdot (2k-1)^{\lceil \frac{n}{N} \rceil - 1} \cdot (2k-1)^{n_1 + n_2 - 1}
$$
choices for $v = r v_1 r^{-1} v_2$ with $|v_1|=n_1$ and $|v_2|=n_2$. 
We note that $n_1 + n_2 + \lceil \frac{n}{N} \rceil = n$, and thus we may crudely estimate that there are
$$
n^2 \cdot (2k)^2 \cdot (2k-1)^{\lceil \frac{n}{N} \rceil - 1} \cdot (2k-1)^{n - 2 \cdot \lceil \frac{n}{N} \rceil}
$$
many choices for~$v = r v_1 r^{-1} v_2$, such that $r = \lceil \frac{n}{N} \rceil$ and where there is at most one cancellation either in $v_1$ or in $v_2$.
Finally, there are $n$ elements that are cyclically conjugate to such a $v$. 

Thus, an upper bound for the total number of words $w \in A^n_N$ that overlap with an inverse may be bounded by
$$
n^3 \cdot (2k)^2  \cdot (2k-1)^{n - \lceil \frac{n}{N} \rceil}.
$$

\item \label{item:no inverse} $w \in A^N_n$ overlaps with a cyclic conjugate of $w$ and $w$ overlaps with itself in a piece $r$ with $|r| =\lceil \frac{n}{N} \rceil$ that does not overlap with itself. Then a cyclic conjugate of $w$ may be written as $v = r v_1 r v_2$ with the conditions on $r$, $v_1$ and $v_2$ as in case~(\ref{item: inverse}.).
We may deduce the same bound by replacing $r^{-1}$ by $r$, if appropriate.

Thus we see that in this case, there are again at most 
$$
n^3 \cdot (2k)^2  \cdot (2k-1)^{n - \lceil \frac{n}{N} \rceil}.
$$
such elements~$w \in A^N_n$.

\item \label{item:no inverse overlap} $w \in A^N_n$ overlaps with a cyclic conjugate of $w$
and $w$ overlaps with itself in a piece $r$ with $|r| = \lceil \frac{n}{N} \rceil$ that overlaps with itself.
Then a cyclic conjugate of $w$ may be written as
$v = r v_1$ and $v = v_2 r v'_2$ for $|r| = \lceil \frac{n}{N} \rceil$, and $|v_2| < |r|$.
where there is at most one cancellation in $v_1$.

Thus, in particular, we have that $v = r v_1'$ and $v = v_2 r$ for $v_1'$ the prefix of~$v_1$ such that $|v_1'| = |v_2|$.

\begin{claim} \label{claim:choices periodical overlapping}
There are at most $2k \cdot (2k - 1)^{M - 1}$ elements $v_2$, $r$, $v_1'$,
such that $r v_1' = v_2 r$, $|v_2| = |v_1'| = M$, $|r| = \lceil \frac{n}{N} \rceil$ and $M < |r|$.
\end{claim}

\begin{proof}
Indeed, we will see that for any choice of $v_2$, the elements $r$ and $v_1'$ are fully determined:
By comparing the first $M$ letters of the equality $r v_1' = v_2 r$
we see that the first $|v_2|$ letters of $r$ are $v_2$.
By continuing this way we see that $r$ has to be a prefix of a power of $v_2$. We know its length, and thus $r$ is determined. We may recover $v_1'$ by evaluating $v_1' = r^{-1} v_2 r$.
There are $2k \cdot (2k - 1)^{M - 1}$ choices of $v_2$, and thus this shows the claim.
\end{proof}

We write $v_1 = v_1' \cdot v_1''$, for $v_1$ and $v_1'$ as above with $|v_1'| = M$.
Note that there is at most one cancellation in $v_1$ and thus at most one cancellation in~$v_1''$.
Thus, there are at most  $n \cdot 2 k \cdot (2k - 1)^{n - \frac{n}{N} - M - 1}$ many choices for~$v_1''$, following the estimate of case~(\ref{item: inverse}.) for words that contain exactly one cancellation.
Together with Claim~\ref{claim:choices periodical overlapping} we see that for $|v_1'| = M$ there are a total of at most
$$
n \cdot (2k)^2 \cdot (2k - 1)^{n - \frac{n}{N} - 2} 
$$
choices for such~$v$. As $M \leq n$ we see that there are at most a total of
$$
n^2 \cdot (2k)^2 \cdot (2k - 1)^{n - \frac{n}{N} - 2} 
$$
choices for~$v$ without any condition on~$v_1'$. 

As at most $n$ words in $A^N_n$ are cyclically conjugate to such a~$v$ we may estimate the total of words in~$A^N_n$ with overlap with the same power in itself by
$$
n^3 \cdot (2k)^2 \cdot (2k - 1)^{n - \frac{n}{N}}.
$$
\end{enumerate}

By putting the estimates from the cases (\ref{item: inverse}.), (\ref{item:no inverse}.), and (\ref{item:no inverse overlap}.) together we deduce that
$$
| A^N_n | \leq 3 n^3 \cdot (2k)^2 \cdot (2k - 1)^{n - \frac{n}{N}}.
$$
This finishes the proof of Proposition~\ref{prop: elts sat small canc}.
\end{proof}

\begin{corr} \label{corr:prob scl}
In the situation of Setup~\ref{setup:random},   
let $q_n$ be the probability that a random element of~$F'_n$ does not
satisfy the small cancellation condition~$C'(1 / \sqrt{n})$. Then
$$
q_n = o\bigl((2k-1)^{- \sqrt{n}/2}\bigr).
$$
\end{corr}

\begin{proof}
Let $B_n \subset F_n$ be the set of elements that do not satisfy the
small cancellation condition~$C'(1 / \sqrt{n})$.  By
Proposition~\ref{prop: elts sat small canc} we see that
$$
|B_n| \leq 3 \cdot n^3 \cdot  (2k)^2 \cdot  (2k-1)^{n -  \lceil \sqrt{n} \rceil}
      = 3 \cdot n^3 \cdot  (2k) \cdot (2k - 1)^{1-\lceil \sqrt{n} \rceil} \cdot |F_n|.
$$
Thus, we may estimate that the probability $q_n$ of a random
element in $F'_n$ to also lie in $B_n$ to be
$$
q_n = \frac{|B_n|}{|F'_n|}
    \leq 3 \cdot n^3 \cdot 2k \cdot (2k - 1)^{1-\lceil \sqrt{n} \rceil} \cdot \frac{|F_n|}{|F'_n|}.
$$
Using Sharp's result (Theorem \ref{thm:sharp}) we can estimate that
$q_n = o((2k-1)^{-\sqrt{n}/2})$ as $n \to \infty$.
\end{proof}

\subsection{Proof of Theorem \ref{thm: random simvol of elts}}

We now give the argument for Theorem~\ref{thm: random simvol of elts}. 

\begin{proof}
  Let $k := |S|$. 
By Theorem~\ref{thm:cal walker rand}, we see that for every $C > 1$,
$\epsilon > 0$ the probability of a random element in $F'_n$ to
satisfy that
\begin{align} \label{eqn:rand elts scl}
\left| \scl(r) \cdot \frac{\log(n)}{n} - \frac{\log(2k-1)}{6} \right| \leq \epsilon
\end{align}
is~$1 - O(n^{-C})$.  By Corollary~\ref{corr:prob scl}, the probability
that a random element in $F'_n$ satisfies the small cancellation
condition $C'(1/ \sqrt{n})$ is $1 - o((2k-1)^{-\sqrt{n}/2})$.  Thus, the
probability that a random element
in~$F'_n$ satisfies Equation~\eqref{eqn:rand elts scl} may be bounded
by~$1 - O(n^{-C}) - o((2k-1)^{-\sqrt{n}/2}) = 1 - O(n^{-C})$.

In the following, let $n \geq 49$. Then,~Theorem \ref{thm:small cancellation}
implies that
$$
4 \cdot \scl(r) \cdot \Bigl(1 - \frac{6}{\sqrt{n}}\Bigr)^{-1}
\leq \|G_r \|
\leq 4 \cdot  \scl(r).
$$
Putting things together we see that if $r$ satisfies Equation~\eqref{eqn:rand elts scl},
then 
\begin{align*}
  \| G_r \| \cdot \frac{\log(n)}{n}
  & \leq 4 \cdot \scl(r) \cdot \frac{\log(n)}{n}
  \\
  & \leq  \frac{2 \log(2k -1)}{3} + 4 \epsilon
\end{align*}
and
\begin{align*}
  \| G_r \| \cdot \frac{\log(n)}{n}
  & \geq 4 \cdot \scl(r) \cdot \frac{\log(n)}{n} \cdot \Bigl(1 - \frac{6}{\sqrt{n}}\Bigr)^{-1}
  \\
  & \geq \frac{2 \log(2k -1)}{3} \cdot \Bigl(1 - \frac{6}{\sqrt{n}}\Bigr)^{-1} - 4 \epsilon \cdot \Bigl(1 - \frac{6}{\sqrt{n}}\Bigr)^{-1}
  \\
  & \geq  \frac{2 \log(2k -1)}{3} - 2 \cdot 4 \epsilon .
  & \text{\hspace{-1.5cm}(because~$n \geq 36$)}
\end{align*}

By relabelling~$\epsilon$ and $C$ we obtain that 
the probability that
$$
\left| \| G_r \| \cdot \frac{\log(n)}{n} - \frac{2 \log(2k - 1)}{3} \right| \leq \epsilon
$$
may be estimated by $1 - O(n^{-C})$.
\end{proof}

%%%%%%%%%%%%%%%%%%%%%%%%%%%%%%%%%%%%%%%%%%%%%%%%%%%%%%%%%%%%%%%%%%%%%%%%

\section{Computational bounds: $\lall$}
\label{sec:lallop}

In this section we describe an invariant for elements in $F(S)'$
called~$\lall$.  This invariant will give a lower bound to the
simplicial volume of one-relator groups and is computable in
polynomial time.

We briefly describe the motivation for the definition of $\lall$. For
this, recall that any element in~$F(S)'$ may be written up to cyclic
conjugation as~$r^M$, where $r$ is root-free and cyclically reduced
and $M \in \N_{>0}$. Throughout this section, we will
write~$\lall(r^M)$, to indicate that $\lall$ is evaluated on a power
of size~$M$, even if the element is root-free, i.e.\ if $M=1$.

Recall that by Proposition \ref{prop:admissible maps and van Kampen
  diagrams} we have that
\[
\| G_{r^M} \| = 
\inf_{\Dcl \in \Delta(r^M)}
                 \frac{-2 \cdot \chi^-(\Sigma_\Dcl)}%
                      {\bigl|n(f_{\Dcl},\Sigma_\Dcl)\bigr|},
\]
where the infimum ranges over all \lall -admissible van~Kampen
diagrams.  Similar to the algorithm \texttt{scallop} \cite{scallop}
that computes stable commutator length of elements in free group, we
wish to compute~$\| G_{r^M} \|$ using a linear programming problem.

Crudely, \texttt{scallop} associates to any $\scl$-admissible surface
a vector in a finite dimensional vector space, spanned
by the finitely many reasonable configurations around the vertices of
$\scl$-admissible surfaces. Then both the Euler characteristic and the
degree of the original van~Kampen diagram may be computed via linear
functions of the associated vector, and thus $\scl$ becomes the
solution of a finite dimensional linear programming problem.

When adapting this algorithm for the simplicial volume of one-relator
groups one runs into the problem that the Euler characteristic for
admissible van~Kampen diagrams may \emph{not} be computed by the
information around the vertices alone. However, we remedy this by
adding an extra term to the Euler characteristic; see Definition
\ref{defn:lall}. This only gives a lower bound of~$\| G_{r^M} \|$ but
allows us to do exact computations. To control this term we will need
to restrict to certain van~Kampen diagrams, which we call reduced
\lall -admissible van~Kampen diagrams (see Remark \ref{rmk: defn of
  lall van Kampen diagrams necessary}).

The aim of this section is to show:

\begin{thm}[$\lall$] \label{thm:lallop}
  Let $S$ be a set and let $r \in F(S)'$. Then
  \begin{enumerate}
  \item \label{item:lb_lallop} $\lall(r) \leq \| G_r \|$,
  \item \label{item:bounded_by_scl_lallop} $\lall(r) \leq 4 \cdot
    \scl_S(r) - 2$, and
  \item \label{item:algorithm_lallop} there is an algorithm to compute
    $\lall(r)$ that is polynomial in~$|r|$, the word length
    of~$r$. Moreover, $\lall(r) \in \Q$.
\end{enumerate}
\end{thm}

We will prove the first two items of Theorem~\ref{thm:lallop} in
Section~\ref{subsec:from vk to lall red vk}. The proof of the third
part will be developed in Sections~\ref{subsec:lpp} and
\ref{subsec:breaking up vertices}.

%%%%%%%%%%%%%%%%%%%%%%%%%%%%
\subsection{$\lall$} \label{subsec:lallop}

We now define reduced \lall-admissible van~Kampen diagrams, which we
will use to define \lall\ in Definition~\ref{defn:lall}.
 
\begin{defn}[$\lall$-admissible van Kampen diagram] \label{defn:lallop admissible tiling}
  Let $r^M$ be a cyclically reduced word with root $r = \xtt_1 \cdots
  \xtt_n \in F(S)' \setminus \{e\}$ and~$M \in \N_{>0}$.  We say that
  a van~Kampen diagram~$\Dcl$ to~$r$ on a surface~$\Sigma_{\Dcl}$ is
  \emph{$\lall$-admissible to~$r^M$}, if every edge is labelled by a
  single letter in~$S$ such that the counterclockwise label around
  every disk~$D \in \Dcl$ is cyclically labelled by~$r^{n \cdot M}$,
  where $n$ is a non-zero integer~$n \in \Z$, called the \emph{degree
    of $D$} and denoted by $n(D)$.
  
 Let $e$ be an oriented edge in the van~Kampen diagram.  This edge is
 adjacent to two disks $D$ and $D'$, where the edge is in
 counterclockwise (positive) orientation for one of the disks and in
 clockwise (negative) orientation in the other.  Thus, if the label of
 the edge~$e$ is $\xtt$, then $\xtt$ labels a subletter of the label
 of~$D$ and a subletter of the inverse of the label of~$D'$.  For~$D$
 we have a \emph{position}~$i \in \{ 1, \ldots, n \}$ and a
 \emph{sign}~$\epsilon \in \{ +1, -1 \}$ corresponding to the
 letter~$\xtt_i$ labelled by that edge in the disk and the sign of the
 degree of the disk. We will write~$i^{\epsilon}$ as a shorthand for
 the position/sign of the edge at a disk and note that
 $\xtt_i^{\epsilon} = \xtt$. Similarly we have a position and sign
 for~$D'$ which we denote by~$i'^{\epsilon'}$ and note that
 $\xtt_{i'}^{\epsilon'} = \xtt^{-1}$.

Note that this way, every oriented edge $e$ in the van~Kampen diagram
has two positions and signs~$(i^{\epsilon}, i'^{\epsilon'})$
corresponding to the two disks the edge is adjacent to. In this case,
$\xtt_i^{\epsilon} = \xtt_{i'}^{-\epsilon'}$. In analogy
to~$\texttt{scallop}$~\cite{Calegari}, we call $(i^{\epsilon},
i'^{\epsilon'})$ the \emph{rectangle} associated to the edge $e$
(Figure~\ref{fig:edgerectangle}).

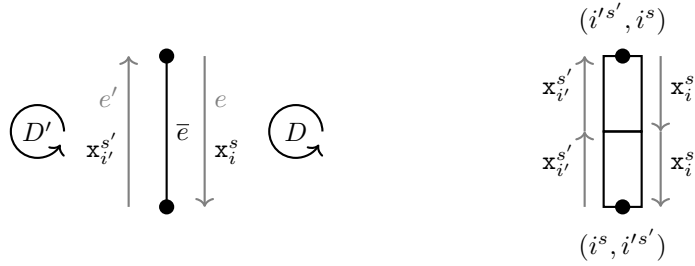
\begin{figure}
  \begin{center}
    \begin{tikzpicture}[x=1cm,y=1cm,thick]
      % edge
      \draw (0,0) -- (0,2);
      \gvertx{(0,0)}
      \gvertx{(0,2)}
      \draw (0,1) node[anchor=west] {$\overline e$};

      \draw (1.7,1) node {$D$};
      \orarrow{(1.7,1)}
      \draw (-1.7,1) node {$D'$};
      \orarrow{(-1.7,1)}
      
      \draw (0.5,0.66) node [anchor=west] {\smash{$\xtt_i^s$}};
      \draw (-0.5,0.66) node [anchor=east] {\smash{$\xtt_{i'}^{s'}$}};
      \begin{scope}[black!50]
        \draw[->] (0.5,2) -- (0.5,0);
        \draw (0.5,1.33) node [anchor=west] {\smash{$e$}};
        \draw[->] (-0.5,0) -- (-0.5,2);
        \draw (-0.5,1.33) node [anchor=east] {\smash{$e'$}};
      \end{scope}
      % rectangle
      \begin{scope}[shift={(6,0)}]
        \gvertx{(0,0)}
        \draw (-0.25,0) rectangle +(0.5,1);
        \draw (0.5,0.5) node [anchor=west] {\smash{$\xtt_i^s$}};
        \draw (-0.5,0.5) node [anchor=east] {\smash{$\xtt_{i'}^{s'}$}};
        \begin{scope}[black!50]
          \draw[->] (0.5,1) -- (0.5,0);
          \draw[->] (-0.5,0) -- (-0.5,1);
        \end{scope}
        \draw (0,-0.5) node {$(i^s, i'{}^{s'})$};
        \gvertx{(0,2)}
        \draw (-0.25,1) rectangle +(0.5,1);
        \draw (0.5,1.5) node [anchor=west] {\smash{$\xtt_i^s$}};
        \draw (-0.5,1.5) node [anchor=east] {\smash{$\xtt_{i'}^{s'}$}};
        \begin{scope}[black!50]
          \draw[->] (0.5,2) -- (0.5,1);
          \draw[->] (-0.5,1) -- (-0.5,2);
        \end{scope}
        \draw (0,2.5) node {$(i'{}^{s'},i^s)$};
      \end{scope}
    \end{tikzpicture}
  \end{center}
  
  \caption{From edges to corresponding rectangles}
  \label{fig:edgerectangle}
\end{figure}

We say that a $\lall$-admissible van~Kampen diagram is \emph{reduced}
if there are no rectangles~$(i^{\epsilon}, i'^{\epsilon'})$ with~$i =
i'$.  We denote the set of reduced $\lall$-admissible van~Kampen
diagrams to~$r^M$ by~$\Deltal(r^M)$.
\end{defn}

We will see that we may always replace a $\lall$-admissible van~Kampen
diagram by a reduced $\lall$-admissible van~Kampen diagram; see
Proposition~\ref{prop:from vk to reduced lall vk}.

\begin{defn}[$\lall$] \label{defn:lall}
Let $s \in F(\Scl)'$ be an element in the commutator subgroup of the
free group~$F(\Scl)$ and let $s$ be conjugate to~$r^M$ where $r$ is
cylically reduced, root-free and~$M \in \N_{>0}$. Then, we define
$$
\lall(r^M) := \inf_{\Dcl \in \Deltal(r^M)} \frac{-2\chi(\Sigma_{\Dcl}) + 2 \sum_{D \in \Dcl}(1- |n(D)|)}{n(\Sigma_{\Dcl})},
$$
where the infimum is taken over all reduced $\lall$-admissible
van~Kampen diagrams for~$r^M$.
\end{defn}

As stated in the introduction, the (de)nominator of the terms in the
definition of~$\lall$ are carefully chosen in such a way that they can
be computed using a linear programming problem similar to
$\texttt{scallop}$.

\begin{rmk} \label{rmk: defn of lall van Kampen diagrams necessary} 
One may wonder why we needed to define reduced \lall-admissible
van~Kampen diagrams in the first place and didn't simply take the
infimum in Definition \ref{defn:lall} over \emph{all} van~Kampen
diagrams.

To see that this is necessary, note that for every word~$r \in F(S)'$
and any natural number~$N \in \N$ we may glue a disk labelled by~$r^N$
to a disk labelled by~$r^{-N}$ by identifying corresponding
letters. Topologically this is a sphere. We may add this sphere to any
van~Kampen diagram over $r$ to obtain a new van~Kampen diagram. While
this does not change the total degree of the van~Kampen diagram, and
only changes the Euler characteristic by $1$, it changes the term $2
\sum_{D \in \Dcl}(1- |n(D)|)$ by $2 - 2 \cdot N$. Thus, an infimum as
in Definition~\ref{defn:lall} over \emph{all} \lall-admissible
van~Kampen diagrams would not exist.

Similarly, if the word we consider is not root-free, say it is of the
form~$r^M$ for some~$M \in \N_{\geq 2}$, we may glue up a disk
labelled by~$r^{N \cdot M}$ with a disk labled by $r^{-N \cdot M}$ by
gluing up corresponding letters shifted by~$r$. Topologically, this
again is a sphere that we may add to any van~Kampen diagram.
\end{rmk}

%%%%%%%%%
\subsection[Towards reduced $\lall$-admissible van~Kampen diagrams]{From van~Kampen diagrams\\ to reduced $\lall$-admissible van~Kampen diagrams} \label{subsec:from vk to lall red vk}

In this section, we show how an arbitrary van~Kampen diagram may be
replaced by a reduced $\lall$-admissible van~Kampen diagram:

\begin{prop} \label{prop:from vk to reduced lall vk}
  Let $r^M \in F(S)$ be a cyclically reduced element where $M \geq 1$
  and $r$ is root-free. Let $\Dcl$ be a van~Kampen diagram
  for~$r^M$. Then, there is a reduced $\lall$-admissible van~Kampen
  diagram~$\Dcl'$ with the same degree, such that
  $-\chi(\Sigma_{\Dcl}) \geq - \chi(\Sigma_{\Dcl'})$.

  Moreover, we have for every reduced $\lall$-admissible van~Kampen
  diagram~$\Dcl$ that $\chi(\Sigma_{\Dcl}) = \chi^-(\Sigma_{\Dcl})$.
\end{prop}

\begin{proof}
  We may assume that every edge of~$\Dcl$ is labelled by some element
  in~$F(S)$ by possibly shrinking the edges that are not labelled by
  any word.  By subdividing the edges, we may further assume that
  every edge is labelled by a single letter in~$S$.  We know that
  every disk is cyclically labelled by a power of~$r^M$. By recording
  which letter of~$r$ corresponds to which edge, we may construct the
  rectangles.

  Suppose that $\Dcl$ is not reduced. Then there is a
  rectangle~$(i^\epsilon, i'^{\epsilon'})$ with~$i = i'$. Since
  $\xtt_i^{\epsilon} = \xtt_{i'}^{-\epsilon'}$, we deduce that
  $\epsilon = - \epsilon'$, in other words, the two disks adjacent to
  this rectangle have opposite signs. We may then cut up the two disks
  at the edge and glue the boundaries together analogously to
  Figure~\ref{fig:inverse_disks_glueing}. This does not change the
  degree and only increases the Euler characteristic.

  We are left to show that $\chi^-$ and $\chi$ agree for reduced
  \lall-admissible van~Kampen diagrams.  If not, there is a reduced
  spherical \lall-admissible van~Kampen diagram~$\Dcl^S$ for~$r^M$,
  such that $\Sigma_{\Dcl^S}$ is a sphere. Note that this would also
  be a van~Kampen diagram for~$r$.  This would then define a
  non-trivial spherical map to the presentation complex. However, the
  presentation complex of root-free words is aspherical by a result of
  Cockcroft~\cite{cockcroft1954two}.
\end{proof}

Using the last proposition, we may prove items \ref{item:lb_lallop}
and \ref{item:bounded_by_scl_lallop} of Theorem~\ref{thm:lallop}.
\begin{proof}[Proof of Theorem~\ref{thm:lallop} items \ref{item:lb_lallop} and
    \ref{item:bounded_by_scl_lallop}]
  The fact that $\lall(r) \leq \| G_r \|$ is a consequence of the
  description of~$\|G_r\|$ in terms of van~Kampen diagrams
  (Proposition~\ref{prop:admissible maps and van Kampen diagrams}) and
  that $\chi$ and $\chi^-$ agree for reduced $\lall$-admissible
  van~Kampen diagrams.  By Proposition~\ref{prop:from vk to reduced
    lall vk}, we may replace van~Kampen diagrams by $\lall$-admissible
  van~Kampen diagrams.
  
  To see item~\ref{item:bounded_by_scl_lallop}, let $\Sigma$ be an
  $\scl$-admissible surface to~$r^M$ with one boundary component.  We
  may assume that we just have one boundary component with positive
  degree~$N$. By gluing in a disk to the boundary we obtain a
  $\lall$-admissible van~Kampen diagram on a surface~$\Sigma'$ with
  $\chi(\Sigma) = \chi(\Sigma') - 1$. We may estimate
  $$
  \lall(r^M) \leq \frac{-2 \chi(\Sigma') + 2 (1 - N) }{N} = \frac{-2 \chi(\Sigma) - 2 N}{N} = 4 \cdot \frac{-2 \chi(\Sigma)}{N} - 2
  $$
  Taking the infimum over all $\scl$-admissible surfaces~$\Sigma$ to~$r^M$, shows
  with the help of the right-hand side that $\lall(r^M) \leq 4 \cdot \scl(r^M) - 2$.
\end{proof}

%%%%%%%%%%%%%%%%%%%%%%%
\subsection[From \dots to linear programming]{From reduced $\lall$-admissible van Kampen diagrams to linear programming} \label{subsec:lpp}

Recall that throughout this section we will write the relator of our
one-relator group as~$r^M$, where $r$ is cyclically reduced and
root-free and $M \in \N_{>0}$. Note that any element in the free group
may be conjugated to an element that can be written in this way.

\begin{prop} \label{prop:computing lallop}
  Let $S$ be a set and let $r^M \in F(S)' \setminus\{e\}$ be
  cyclically reduced such that $r$ is root-free. Then $\lall(r^M)$ can
  be computed via the information around the vertices as follows:
  \begin{eqnarray*}
    \lall(r^M) 
    &=& \inf_{\Dcl \in \Deltal(r^M)} \frac{  \sum_{v \in V_\Dcl} \bigl( \deg(v)-2 \bigr)   - 2 \cdot \sum_{D\in \Dcl} |n(D)|}{\sum_{D\in \Dcl} n(D)}
  \end{eqnarray*}
  Here, we write~$V_\Dcl$ for the set of vertices of a van~Kampen diagram~$\Dcl$.
\end{prop}

\begin{proof}
  Observe that $\sum_{v \in V_\Dcl} \bigl( \deg(v)-2 \bigr)$ is equal
  to $-2 \cdot \left( \# \text{vertices} - \# \text{edges} \right)$.
  Similarly, $2 \cdot \sum_{D\in \Dcl} |n(D)| = -2 \# \text{faces} + 2
  \sum_{D \in \Dcl}(1- |n(D)|)$.  Putting things together we see that
  $$
 \sum_{v \in V_\Dcl} \bigl( \deg(v)-2 \bigr)   - 2 \cdot \sum_{D\in \Dcl} |n(D)| = -2\cdot \chi(\Sigma_{\Dcl}) + 2 \cdot \sum_{D \in \Dcl}(1- |n(D)|).
  $$
  Thus, the result follows from the definition of $\lall$-admissible
  van~Kampen diagrams (Definition \ref{defn:lallop admissible tiling}).
\end{proof}

A key observation will be that $\lall(r^M)$ may be computed ``locally'' by
computing the degrees of the vertices in the van Kampen diagram. In
contrast, it is impossible to compute the Euler characteristic of the
underlying surface in this way because the information of how large the
disks are cannot be encoded in the vertices.

In a first step, we will associate to a reduced \lall-admissible van~Kampen
diagram a vector in an infinite dimensional vector space by encoding
the local compatibility conditions around the vertices. We can then
compute~$\lall$ as an affine function on this vector space.  Moreover,
we will characterise all vectors that arise in this correspondence
(Lemma~\ref{lem:diagramsasvectors}).

Let $r^M \in F(S)' \setminus\{e\}$ be such that $r$ is cyclically
reduced and root-free with $M \in \N_{>0}$ and write~$r = \xtt_0
\cdots\xtt_{n-1}$. Let $\Dcl \in \Deltal(r^M)$ be a reduced
$\lall$-admissible van~Kampen diagram.

Recall (Definition \ref{defn:lallop admissible tiling}) that for every
oriented edge~$\bar{e}$ in the reduced \lall-admissible van~Kampen
diagram~$\Dcl$ we associate a pair~$(i^\epsilon, {i'}^{\epsilon'})$,
called rectangle, as follows: The integers~$i, i'$ correspond to the
letters~$\xtt_i$ and $\xtt_{i'}$ that label the sides $e$ and $e'$
of~$\bar{e}$ and the signs~$\epsilon, \epsilon' \in \{ +1, -1 \}$
correspond to the disks adjacent to~$\bar{e}$.  We denote this
rectangle by~$\Rec(\bar{e})$. As $\Dcl$ is reduced we know that~$i
\not = i'$.

By abuse of notation we denote by $\Rec(r^M)$ the set
$$
  \Rec(r^M) := \bigl\{ (i^s, i'^{s'}) \bigm| i \not = i' \in \{ 0, \ldots, n-1 \},\  s, s' \in \{+,- \},\ \xtt_i^s = \xtt_{i'}^{-s'} \bigr\}
$$
of all possible rectangles.

Observe that if $\bar{e}'$ is the inverse of the oriented edge
$\bar{e}$ and if $\Rec(\bar{e}) = (i^\epsilon, {i'}^{\epsilon'})$,
then $\Rec(\bar{e}') = ({i'}^{\epsilon'}, i^\epsilon)$. This defines
an involution $\iota \col \Rec(r^M) \to \Rec(r^M)$ on the set of
rectangles and we think of~$\iota$ as flipping the orientation of the
edge (Figure~\ref{fig:edgerectangle}).

We now turn to the structure around a vertex: Let $v$ be a vertex
of~$\Dcl$ and let $\bar{e}_1, \ldots, \bar{e}_k$ be the edges
in~$\Dcl$ pointing towards~$v$ and ordered clockwise around~$v$.  
We associate the tuple
$$
\Ver(v) := \bigl[\Rec(\bar{e}_1), \ldots,  \Rec(\bar{e}_k)\bigr].
$$
to $v$.  The tuples of rectangles arising in this way are not
arbitrary, as they have to be compatible with the labelling of the
disks in~$\Dcl$.  More precisely: Let $\bar{e}_1$ have the sides $e_1$
and~$e'_1$ with rectangle~$\Rec(\bar{e}_1) = (i_1^{\epsilon_1},
{i'}_1^{\epsilon'_1})$ and let $\bar{e}_2$ have the sides $e_2$
and~$e'_2$ with rectangle~$\Rec(\bar{e}_2) = (i_2^{\epsilon_2},
{i'}_2^{\epsilon'_2})$; see Figure~\ref{fig:localpod}. Then the
disk~$D$ adjacent to the edges $\bar{e}_1$ and $\bar{e}_2$ has to have
the same sign, i.e., $\epsilon'_1 = \epsilon_2$.

Suppose that $\epsilon'_1 = \epsilon_2 = +$. The labels of the
disk~$D$ thus read cyclically a positive power of~$r$. Hence, if the
label of~$\bar{e}_1$ for~$D$ is $\xtt_{i'_1}$, the label
of~$\bar{e}_2$ is $\xtt_{i_2}$ and $i_2 = i'_1 + 1$.  Similarly, if
the sign of~$D$ were negative we would have $i_2 = i_1'-1$, where all
indices are taken modulo~$n$; see Figure \ref{fig:localpod}.

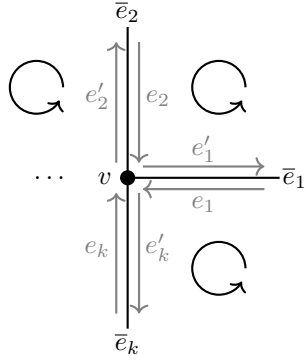
\begin{figure}
  \begin{center}
    \begin{tikzpicture}[x=1cm,y=1cm,thick]
      % local structure
      \gvertx{(0,0)}
      \draw (0,0) -- (0,2);
      \draw (2.2,0) node {$\overline e_1$};
      \draw (0,0) -- (2,0);
      \draw (0,2.2) node {$\overline e_2$};
      \draw (-1,0) node {\dots};
      \draw (0,0) -- (0,-2);
      \draw (0,-2.2) node {$\overline e_k$};
      \orarrow{(1.2,1.2)}
      \orarrow{(1.2,-1.2)}
      \orarrow{(-1.2,1.2)}
      \draw (-0.3,0) node {$v$};
      \begin{scope}[black!50]
        \draw[->] (0.2,0.15) -- (1.8,0.15);
        \draw (1,0.3) node {\smash{$e_1'$}};
        \draw[->] (1.8,-0.15) -- (0.2,-0.15);
        \draw (1,-0.4) node {\smash{$e_1$}};

        \begin{scope}[rotate={90}]
        \draw[->] (0.2,0.15) -- (1.8,0.15);
        \draw (1,0.4) node {\smash{$e_2'$}};
        \draw[->] (1.8,-0.15) -- (0.2,-0.15);
        \draw (1,-0.4) node {\smash{$e_2$}};
        \end{scope}

        \begin{scope}[rotate={270}]
        \draw[->] (0.2,0.15) -- (1.8,0.15);
        \draw (1,0.4) node {\smash{$e_k'$}};
        \draw[->] (1.8,-0.15) -- (0.2,-0.15);
        \draw (1,-0.4) node {\smash{$e_k$}};
        \end{scope}
      \end{scope}
    \end{tikzpicture}
  \end{center}
 
  \caption{The local structure around a vertex}
  \label{fig:localpod}
\end{figure}

This motivates the following definition: We say that \emph{a rectangle
  $({i_1}^{s_1}, {i_1'}^{s_1'}) \in \Rec(r)$ follows the rectangle}
$({i_2}^{s_2}, {i_2'}^{s_2'}) \in \Rec(r)$ if
\begin{itemize}
\item $i'_1 = i_{2}+1$ and $s'_1 = + = s_{2}$, or
\item $i'_1 = i_{2}-1$ and $s'_1 = - = s_{2}$.
\end{itemize}

A tuple~$[R_1, \dots, R_k]$ of rectangles~$R_1,\dots, R_k \in \Rec(r)$
is a \emph{$k$-pod} if $R_{i+1}$ follows~$R_i$ for all~$i \in
\{1,\dots, k-1 \}$ and $R_1$ follows $R_k$.  We then define
$$
\Ver(r^M) = \bigl\{ [R_1, \ldots, R_k] \bigm| k \in \N_{\geq 2},\  R_1, \dots, R_k \in \Rec(r),\
          \text{$[R_1,\dots, R_k]$ is a $k$-pod} 
          \bigr\}.
$$          
By construction, if $v$ is a vertex of~$\Dcl$, then $\Ver(v) \in
\Ver(r^M)$. The set~$\Ver(r^M)$ is infinite and should be thought of as
the set of all possible labels around a vertex in a $\lall$-admissible
van~Kampen diagram to~$r^M$.
          
We illustrate this by the following example:

\begin{exmp}
  Let $r = \att \btt \att^{-1} \btt^{-1}$ be the commutator of $\att$
  and $\btt$ in $F(\att, \btt)$. In the above setting, $\xtt_1=\att$,
  $\xtt_2 = \btt$, $\xtt_3 = \att^{-1}$ and $\xtt_4 = \btt^{-1}$.
  Then
  \begin{align*}
    \Rec(r) =   \bigl\{ & (1^+,3^+), (3^+,1^+), (1^-,3^-),(3^-,1^-), \\
    & (2^+,4^+), (4^+,2^+), (2^-,4^-),(4^-,3^-) \bigr\}.
  \end{align*}
  Examples of $4$-pods are 
  \begin{align*}
    & [(1^+,3^+),(4^+,2^+),(3^+,1^+),(2^+,4^+)] \quad\text{or} \\
    & [(1^-,3^-),(2^-,4^-),(3^-,1^-),(4^-,2^-)]. 
  \end{align*}
\end{exmp}

Let $\Z\Ver(r^M)$ be the $\Z$-module freely generated by~$\Ver(r^M)$.  We
will now define a map~$\Phi \colon \Deltal(r) \longrightarrow
\Z\Ver(r^M)$ encoding the local structure of van~Kampen diagrams: For a
$\lall$-admissible van Kampen diagram~$\Dcl$ on a
surface~$\Sigma_\Dcl$ with vertex set~$V_\Dcl$, we set
$$
\Phi(\Dcl):= \sum_{v \in V_{\Dcl}} \Ver(v) \in \Z\Ver(r^M).
$$

Let $\A(r^M) \subset \Z V(r^M)$ be the set of all elements in~$\N\Ver(r^M)
\subset \Z\Ver(r^M)$ such that for each rectangle~$R \in \Rec(r^M)$, the
number of occurrences of~$R$ coincides with the number of occurrences
of the flipped rectangle~$\iota(R)$. The set~$\A(r^M) \subset \Z\Ver(r^M)$
is defined by a finite set of integral linear equations and
inequalities.  Furthermore, we will consider the corresponding
rational version
\[ \A^\Q(r^M) \subset \Q\Ver(r^M),
\]
which is defined by the corresponding (in)equalities.

\begin{lemma}\label{lem:diagramsasvectors}
  We have~$\Phi(\Deltal(r^M)) \subset \A(r^M)$. For every $a \in
  \A(r^M)$, there is a van~Kampen diagram~$\Dcl$ such that $\Phi(\Dcl)
  = 2 \cdot M \cdot a$.
\end{lemma}

For the proof we need the following result, originally found in
Neumann's article~\cite{neumann2001immersed}.

\begin{lemma}[\protect{\cite[Lemma 3.2]{neumann2001immersed}}] \label{lemma:neumann cover}
  Let $\Sigma$ be an oriented surface of positive genus. Let $N \in
  \N$ be an integer and suppose that for every boundary component
  of~$\Sigma$ there is a collection of degrees summing to $N$.  Then
  there is a connected $N$-fold covering~$\Sigma'$ of~$\Sigma$ with
  prescribed degrees on the boundary components over each boundary
  component of~$\Sigma$ if and only if the prescribed number of
  boundary components of the cover has the same parity as~$N \cdot
  \chi(\Sigma)$.
\end{lemma}

\begin{proof}[Proof of Lemma~\ref{lem:diagramsasvectors}]
  By construction, $\Phi(\Deltal(r^M)) \subset \A(r^M)$.  To prove
  that for every~$a \in \A(r^M)$ we have that $2 \cdot M \cdot a \in
  \Phi(\Deltal(r^m))$, we first assume that $M=1$.  In this case,
  every element of~$\A(r)$ gives rise to a $\lall$-admissible van
  Kampen diagram over $r$ as follows: Represent pods geometrically by
  stars.  We then choose a matching for the rectangles related by
  flipping and use this to construct the $1$-skeleton by gluing the
  corresponding rectangles of the pods with opposite orientations. We
  now use the ordering of the rectangles in the pods to glue in
  $2$-disks (whose labels will be non-trivial powers of~$r^M$ because
  the rectangles in the pods are following each other). The resulting
  $2$-dimensional CW-complex is homeomorphic to an orientable closed
  connected surface~\cite[p.~87]{moharthomassen} with a \lall
  -admissible van~Kampen diagram $\Dcl$ coming from the labels of the
  disks. It is easy to see that $\Phi(\Dcl) = a$. By taking two copies
  of $\Dcl$ we conclude the proposition for $M=1$.
  
  Now suppose that $M > 1$ and let $a \in \A(r^M)$.  As $\A(r^M) =
  \A(r)$, we can first construct a reduced $\lall$-admissible
  van-Kampen diagram~$\Dcl$ for~$\A(r)$. Let $\Sigma_{\Dcl}$ be the
  associated surface to~$\Dcl$ and let $\Sigma^\partial_{\Dcl}$ be the
  surface~$\Sigma_{\Dcl}$ with the van~Kampen diagrams removed. Thus
  $\Sigma^\partial_{\Dcl}$ is a surface with $|\Dcl|$ many boundary
  components. We will use the following claim:
  
\begin{claim}[coverings of surfaces] \label{claim:covers of surfaces}
  Let $\Sigma$ be a surface with $p$ boundary components and let $M
  \in \N$. Then there is a covering~$\Sigma'$ of~$\Sigma$, such that
  each boundary component has degree precisely~$M$.
\end{claim}

\begin{proof}
  We use Lemma~\ref{lemma:neumann cover}. Suppose that $\Sigma$ has
  $p$ boundary components. Choose for every boundary of~$\Sigma$ two
  degree~$M$ boundaries as in the statement of
  Lemma~\ref{lemma:neumann cover}. The resulting map has degree~$N = 2
  \cdot M$ and it has $2 \cdot p$ boundary components. This is an even
  number, just as~$2 \cdot N \cdot \chi(\Sigma)$.  Thus, there is a $2
  \cdot M$-covering~$\Sigma'$ of~$\Sigma$ with the desired property.
\end{proof}

  Using Claim~\ref{claim:conditions for lemma}, we see that there is a
  $2 \cdot M$-covering~${\Sigma'}^{\partial}$
  of~$\Sigma^{\partial}_{\Dcl}$, where each of the $2 \cdot p$
  boundaries maps with degree~$M$.  We may now fill in all the
  boundaries of~${\Sigma'}^{\partial}$ with disks and pull back the
  labels from~$\Sigma_{\Dcl}$. This describes a van~Kampen
  diagram~$\Dcl'$ on a surface~$\Sigma'$ with $2 | \Dcl|$ disks
  over~$r^M$. We see that every vertex of~$\Dcl$ corresponds to $2
  \cdot M$ vertices of~$\Dcl'$ under the covering and that the labels
  around the vertices are identical.  Thus, $\Phi(\Dcl') = 2 M \cdot
  a$, as claimed in the proposition.
\end{proof}  

We will now express the (de)nominators in the computation
of~$\lall(r^M)$ in Proposition~\ref{prop:computing lallop} by suitable
linear maps on~$\Z\Ver(r^M)$.  For a rectangle~$R$, let $s_1(R), s_2(R)
\in \{\pm1\}$ denote the signs of the first and second component,
respectively.  We define the following $\Z$-linear maps:
\begin{align*}
  \nu \colon \Z\Ver(r^M)
  & \longrightarrow \Q \\
  %\Ver(r^M) \ni
  [R_1,\dots, R_k]
  & \longmapsto
  \frac1{M \cdot |r|} \cdot \sum_{i=1}^k \bigl(s_1(R_i) + s_2(R_i)\bigr)
  \\
  \bar\nu \colon \Z\Ver(r^M)
  & \longrightarrow \Q \\
  %\Ver(r^M) \ni
  [R_1,\dots, R_k]
  & \longmapsto
  \frac1{M \cdot |r|} \cdot 2 \cdot k
  \\
  \lambda \colon \Z\Ver(r^M)
  & \longrightarrow \Q \\
  %\Ver(r^M) \ni
  [R_1,\dots, R_k]
  & \longmapsto 
  \frac12 \cdot (k - 2)
\end{align*}

\begin{lemma}
  If $\Dcl \in \Deltal(r)$, then
  \begin{align*}
    \nu\bigl(\Phi(\Dcl)\bigr)
    & = \sum_{D \in \Dcl} n(D)
    \\
    \bar\nu\bigl(\Phi(\Dcl)\bigr)
    & = \sum_{D \in \Dcl} |n(D)|
    \\
    \lambda\bigl(\Phi(\Dcl)\bigr)
    & = \frac12 \cdot \sum_{v \in V_\Dcl} \bigl( \deg(v) - 2\bigr).
  \end{align*}
\end{lemma}
\begin{proof}
  For $\nu$ and $\bar\nu$ we only need to note that every occurrence
  of~$r$ will be counted $|r|$ times when counting the two edges of
  all rectangles (with or without signs). As vertices of degree~$k$
  in~$\Dcl$ are modelled by $k$-pods, the claim for~$\lambda$ follows.
\end{proof}

\begin{prop} \label{prop:lallop rational but inefficient}
  Let $S$ be a set and let $r \in F(S)'\setminus\{e\}$. Then
  $\lall(r)$ is the solution of an infinite linear programming problem
  that is defined over~$\Q$.
\end{prop}

\begin{proof}
Using Proposition~\ref{prop:computing lallop} and
Lemma~\ref{lem:diagramsasvectors}, we see that
\begin{align*}
  \lall(r)
  & = \inf_{\Dcl \in \Deltal(r)} \frac{  \sum_{v \in V_\Dcl} \bigl( \deg(v)-2 \bigr)
    - 2 \cdot \sum_{D \in \Dcl} |n(D)|}{\sum_{D\in \Dcl} n(D)}
  \\
  & = \inf_{a \in \A(r)} 2 \cdot \frac{\lambda(a)- \bar{\nu}(a)}{\nu(a)}
  %\\
  %&
  = \inf_{\substack{a \in \A(r)\\ \nu(a) \geq 1}} 2 \cdot \frac{\lambda(a)- \bar{\nu}(a)}{\nu(a)}
\end{align*}
The function on the right-hand side is invariant under
scaling. Because the (de)nominator is Lipschitz continuous, we
conclude that
\[ \lall(r)
   = \inf_{\substack{a \in \A^\Q(r)\\ \nu^\Q(a) \geq 1}} 2 \cdot \frac{\lambda^\Q(a)- \bar{\nu}^\Q(a)}{\nu^\Q(a)},
\]
where $\nu^\Q$, $\bar\nu^\Q$, and $\lambda^\Q$ are the rational
extensions of the corresponding functions on~$\A(r)$. Hence,
$\lall(r)$ is the solution of an infinite fractional linear
programming problem that is defined over~$\Q$. Applying the
Charnes-Cooper transformation, shows that $\lall(r)$ is also the
solution of a corresponding infinite linear programming problem that
is defined over~$\Q$.
\end{proof}

%%%%%%%%%%%%%%%%%
\subsection{Breaking up the pods: A polynomial algorithm}
\label{subsec:breaking up vertices}

Finally, we reduce the linear programming problem of
Proposition~\ref{prop:lallop rational but inefficient} to a finite
linear programming problem (defined over~$\Q$), which allows to
compute~$\lall$ in polynomial time.  This will be achieved by
``breaking up'' the elements in~$\Ver(r^M)$ into finitely many types of
pod-like configurations with two or three edges, which in turn are
related by linear equations.

For this we first define abstract \emph{pairs} of rectangles: 
\[ \RP(r^M) := \{ (R_1, R_2) \mid R_1, R_2 \in \Rec(r^M)\}.
\]
These rectangles will represent ``open'' parts in pod fragments.
Furthemore, we define the following sets (Figure~\ref{fig:breakup}):
\begin{itemize}
\item $\BP(r^M) \subset \Ver(r^M)$,  the set of all $2$-pods, called \emph{bipods},
\item $\TP(r^M) \subset \Ver(r^M)$,  the set of all $3$-pods, called \emph{tripods},
\item $\OTP(r^M) := \OTP_1(r^M) \sqcup \OTP_2(r^M)$,
  where
  \begin{align*}
    \OTP_1(r^M) &:= 
    \{ \otp{R}{R_1}{R_2} \mid R,R_1,R_2\in \Rec(r^M),\ \text{$R_2$ follows~$R$ and $R$ follows~$R_1$}\},
    \\
    \OTP_2(r^M) &:= 
    \{ \otpp{R}{R_1}{R_2} \mid R,R_1,R_2\in \Rec(r^M),\ \text{$R_2$ follows~$R$ and $R$ follows~$R_1$}\},
  \end{align*}
  the set of \emph{open tripods} (they are open ``between $R_1$ and~$R_2$''),
\item $\DOTP(r^M) = \{ \dotp{R}{R_1}{R_2} \mid R,R_1,R_2 \in
  \Rec(r^M),\ \text{$R_2$ follows $R_1$}\}$, the set of \emph{doubly
    open tripods} (they are open ``between $R$ and $R_1$'' and
  ``between $R_2$ and $R$'').
\end{itemize}

\begin{figure}
  \begin{center}
    \begin{tikzpicture}[x=1cm,y=1cm,thick]
      % initial otp
      \draw (0,0) -- +(0,1);
      \draw (0,0) -- +(30:1);
      \draw (0,0) -- +(-30:1);
      \draw (0,1.2) node {$R_k$};
      \draw (30:1.2) node {$R_1$};
      \draw (-30:1.2) node {$R_2$};
      \begin{scope}[black!50]
        \draw[->] (80:0.5) arc(80:40:0.5);
        \draw[->] (20:0.5) arc(20:-20:0.5);
      \end{scope}
      \gvertx{(0,0)}
      \draw (0,-2) node {\small$\otp{R_1}{R_k}{R_2}$};
      % dotp
      \begin{scope}[shift={(-3,0)}]
        \draw (0,0) -- +(0,1);
        \draw (0,0) -- +(330:1);
        \draw (0,0) -- +(270:1);
        \draw (0,1.2) node {$R_k$};
        \draw (330:1.2) node {$R_2$};
        \draw (270:1.2) node {$R_3$};
        \begin{scope}[black!50]
          \draw[->] (320:0.5) arc(320:280:0.5);
        \end{scope}
        \gvertx{(0,0)}
        \draw (0,-2) node {\small$\dotp{R_k}{R_2}{R_3}$};
      \end{scope}
      % dotp
      \begin{scope}[shift={(-6,0)}]
        \draw (0,0) -- +(0,1);
        \draw (0,0) -- +(270:1);
        \draw (0,0) -- +(210:1);
        \draw (0,1.2) node {$R_k$};
        \draw (270:1.2) node {$R_i$};
        \draw (210:1.2) node {$R_{i+1}$};
        \begin{scope}[black!50]
          \draw[->] (260:0.5) arc(260:220:0.5);
        \end{scope}
        \gvertx{(0,0)}
        \draw (0,-2) node {\small$\dotp{R_k}{R_i}{R_{i+1}}$};
      \end{scope}

      % final otp
      \begin{scope}[shift={(-9,0)}]
        \draw (0,0) -- +(0,1);
        \draw (0,0) -- +(150:1);
        \draw (0,0) -- +(210:1);
        \draw (0,1.2) node {$R_k$};
        \draw (150:1.2) node {$R_{k-1}$};
        \draw (210:1.2) node {$R_{k-2}$};
        \begin{scope}[black!50]
          \draw[->] (200:0.5) arc(200:160:0.5);
          \draw[->] (140:0.5) arc(140:100:0.5);
        \end{scope}
        \gvertx{(0,0)}
        \draw (0,-2) node {\small$\otpp{R_{k-1}}{R_{k-2}}{R_k}$};
      \end{scope}
    \end{tikzpicture}
  \end{center}
    
  \caption{Breaking up a $k$-pod into (doubly) open tripods}
  \label{fig:breakup}
\end{figure}
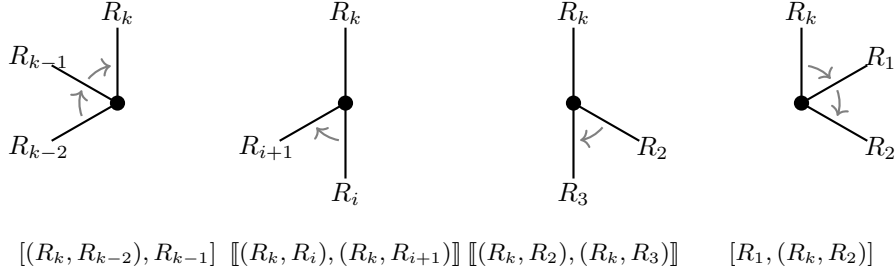  

We will now break up pods into these building blocks
(Figure~\ref{fig:breakup}).  Let $\B(r^M)$ be the free $\Z$-module
freely generated by the disjoint union
\[ \BP(r^M) \sqcup \TP(r^M) \sqcup \OTP(r^M) \sqcup \DOTP(r^M)
\]
and let $\B^\Q(r^M) := \Q \otimes_\Z \B(r^M)$.  Clearly, $\B^\Q(r^M)$ is
finite dimensional. We then consider the $\Q$-linear decomposition map
\begin{align*}
  \Phi_0 \colon \Q\Ver(r^M) & \longrightarrow \B^\Q(r^M) \\ 
       [R_1, \dots, R_k]
       & \longmapsto
       \begin{cases}
         [R_1,\dots, R_k] & \text{if $k \in \{2,3\}$}
         \\
           \otp{R_1}{R_k}{R_2}
           + \sum_{i=2}^{k-3} \dotp{R_k}{R_i}{R_{i+1}}
           +\otpp{R_{k-1}}{R_{k-2}}{R_k}
         & \text{if $k \geq 4$.}
       \end{cases}
\end{align*}
If $x \in \Ver(r^M)$ and $(R_1, R_2) \in \RP(r^M)$, then the number of
occurrences of the pair~$(R_1,R_2)$ in~$\Phi_0(x)$ in the first component of
(doubly) open tripods coincides with the number of occurrences
of~$(R_1,R_2)$ in the second component. We define the subset~$\A_0(r^M)
\subset \B(r^M)$ as the set of all elements such that
\begin{enumerate}
\item all coefficients are non-negative and
\item for every $R \in \Rec(r^M)$, the number of occurrences of~$R$
  equals the number of occurrences of~$\iota(R)$ and
\item for every $P \in \RP(r^M)$, the number of occurences of~$P$ in the
  first component equals the
  number of occurrences of~$P$ in the second component.
\end{enumerate}
Furthermore, we consider the corresponding rational
version~$\A^\Q_0(r^M) \subset \B^\Q(r^M)$.  By construction,
$\Phi_0(\A^\Q(r^M)) \subset \A_0^\Q(r^M)$. Conversely, by matching up
rectangle pairs in the first/second component, we see that
$\Phi_0(\A^\Q(r^M)) = \A_0^\Q(r^M)$.

The functions~$\lambda$, $\nu$, and $\bar\nu$ can be translated to
functions~$\lambda_0, \nu_0, \bar\nu_0 \colon \B^\Q(r^M) \longrightarrow
\Q$ as follows: On elements of~$\BP(r^M) \sqcup \TP(r^M)$, we define them
as before.
\begin{itemize}
\item If $v = \otp{R}{R_1}{R_2} \in \OTP_1(r^M)$ or $v = \otpp{R}{R_1}{R_2}\in \OTP_2(r^M)$, then
  \[ \lambda_0(v) := \frac12, \quad
     \bar\nu_0(v) := \frac4{|r| \cdot M},\quad 
     \nu_0(v) := \frac{s_1(R_1) + s_2(R) + s_1(R) + s_2(R_2)}{|r| \cdot M}.
  \]
%  Here $s_1(R)$ and $s_2(R)$ are the signs of the first and second component
%  of a rectangle~$R$, respectively.
\item If $v = \dotp{R}{R_1}{R_2} \in \DOTP(r^M)$,
  then
  \[ \lambda_0(v) := \frac12, \quad
  \bar\nu_0(v) := \frac2{|r| \cdot M}, \quad
  \nu_0(v) := \frac{s_1(R_1) + s_2(R_2)}{|r| \cdot M}.
  \]
\end{itemize}
A straightforward computation shows that
\[ \lambda_0 \circ\Phi_0 = \lambda^\Q, \quad
\bar\nu_0 \circ \Phi_0 = \bar\nu^\Q, \quad
\nu_0 \circ \Phi_0 = \nu^\Q.
\]

We can now complete the \emph{proof of
  Theorem~\ref{thm:lallop}.\ref{item:algorithm_lallop}}: By
Proposition~\ref{prop:lallop rational but inefficient} and the
previous considerations, we have
\[ \lall(r^M)
= \inf_{\substack{a \in \A_0^\Q(r^M)\\\nu_0(a) \geq 1}}
  2 \cdot \frac{\lambda_0(a) - \bar\nu_0(a)}{\nu_0(a)}.
\]
Thus it suffices to solve the (fractional) linear programming problem
on $\A_0^\Q(r^M)$. The linear cone $\A_0^\Q(r^M)$ has only polynomial
dimension (namely of order~$\mathcal{O}(|r|^5)$) and via the
Charnes-Cooper transform this corresponds to a linear programming
problem in the same order of dimension.  In particular, $\lall(r^M)
\in \Q$, because everything is defined over~$\Q$.  There are now
several available methods to compute the exact value of a linear
programming problem, for example, the algorithm by
Karmarkar~\cite{Karmarkar}.  Thus, there is an algorithm that
determines~$\lall(r^M)$ in polynomial time in~$|r| \cdot M$).  This
finishes the proof of Theorem~\ref{thm:lallop}.

%%%%%%%%%%%%%%%%%%%
\subsection{Examples} \label{subsec:examples lallop}

We implemented the algorithm skeleton~$\lall$ described in the
previous section in~MATLAB~\cite{lallop} and in Haskell~\cite{lallophs}. Thus,
we have a polynomial time algorithm to compute lower bounds for~$\|
G_r \|$.  Upper bounds, on the other hand, may be computed by finding
an explicit van Kampen diagram on a surface for this relator~$r$.

We will illustrate this by an example, whose stable commutator length
was studied by
Calegari~\cite[Section~4.3.5]{Calegari}\cite{Calegari_sss}: For all~$m
\in \N_{\geq2}$, we have
$$
\scl_{\{ \att, \btt \}}(r_m) = \frac{2m-3}{2m-2}, 
$$
where $r_m := [\att, \btt] [\att, \btt^{-m}]$.

\emph{An upper bound for~$\|G_{r_m}\|$.} 
Calegari~\cite[Section 4.3.5]{Calegari} described a van Kampen
diagram~$\Dcl_m$ on a surface~$\Sigma_m$ of genus~$m-1$ with $2m-2$
positive disks that are labelled by the word $r_m = [\att, \btt] [\att,
  \btt^{-m}]$.  Thus, using Proposition~\ref{prop:admissible maps and
  van Kampen diagrams}, we see that
$$
\| G_{r_m} \| \leq \frac{- 2 \cdot \chi(\Sigma_m)}{n(f_{\Dcl_m}, \Sigma_{\Dcl_m})}
= \frac{2m-4}{m-1} = 4 \cdot \Bigl( \scl_{\{ \att, \btt \}}(r) - \frac{1}{2}\Bigr).
$$

We now describe the explicit van Kampen diagram for the case of~$r_3 =
\att \btt \att^{-1} \btt^{-1} \att \btt^{-3} \att^{-1} \btt^{-3}$; the
resulting surface $\Sigma_3$ will have genus~$2$ and the van Kampen
diagram will consist of four disks.  Let us consider
Figure~\ref{fig:l1example}, where $x_i$ is glued to~$X_i$ for all~$i
\in \{1, \ldots, 13 \}$. We may check that the result is a surface of
genus~$2$. We will label the edges by group elements.  For an oriented
edge~$x$ we will denote the label by $\omega(x)$. If $X$ is the
inverse of $x$ then we require that $\omega(X) = \omega(x)^{-1}$.  We
set:
$$
\begin{pmatrix}
\omega(x_1) = \texttt{bA} & \omega(x_2) = \texttt{b} & \omega(x_3) = \texttt{bb} & \omega(x_4) = \texttt{AB} \\
\omega(x_5) = \texttt{B} & \omega(x_6) = \texttt{a} & \omega(x_7) = \texttt{A} & \omega(x_8) = \texttt{BA} \\
\omega(x_9) = \texttt{Ab} & \omega(x_{10}) = \texttt{Ab} & \omega(x_{11}) = \texttt{BA} & \omega(x_{12}) = \texttt{bbb} \\
\omega(x_{13}) = \texttt{bbb} &  &  & 
\end{pmatrix}.
$$
We see that this indeed describes an $l^1$-admissible van Kampen
diagram for $r_3$. All of the disks $D_1, D_2, D_3$ and $D_4$ are
cyclically labelled by~$r$. For example the boundary of~$D_1$ is
(anticlockwise) $x_{10}, x_2, X_4, X_5, X_6, X_9, X_{13}$, where capitalization of letters corresponds to the inverse of the lower case label. Thus the
boundary label is $\texttt{Ab} \cdot \texttt{b} \cdot \texttt{ba}
\cdot \texttt{b} \cdot \texttt{A} \cdot \texttt{Ba} \cdot \texttt{BBB}
= \texttt{AbbbabABaBBB}$, which is a cyclic conjugate of~$r$.  The
result is a van Kampen diagram on a surface of genus $2$.

\begin{figure} 
  \centering
  \includegraphics[scale=0.8]{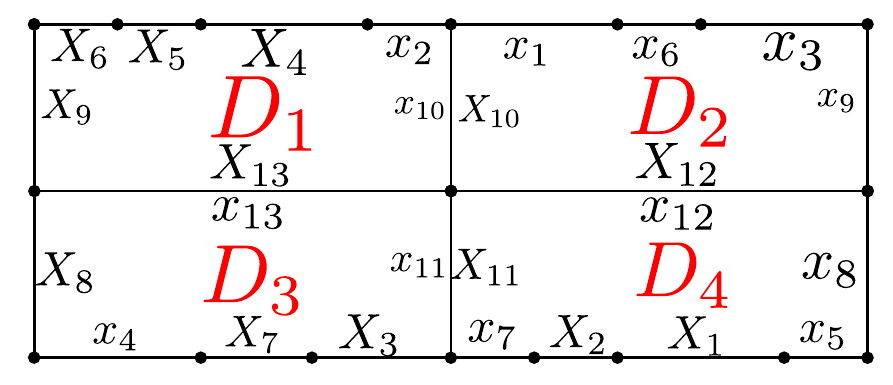}
  \caption{An $l^1$-admissible van Kampen diagram on a surface
    $\Sigma_{\Dcl}$ with $\chi(\Sigma_{\Dcl}) = -2$ and degree $4$.}
  \label{fig:l1example}
\end{figure}

\emph{A lower bound for~$\|G_{r_m}\|$.} 
On the other hand, we may compute lower bounds of $\| G_{r_m} \|$
using the algorithm described in the previous
section~\cite{lallop,lallophs}. In this way, we obtained the values
$\lall(r_2)=0$, $\lall(r_3)=1$, and $\lall(r_4)=\frac{4}{3}$.  We were
not able to compute~$\lall(r_i)$ for larger~$i$ since the linear
programming problem involved in the solution of~$\lall$ becomes too
large.  Using Theorem~\ref{thm:lallop} and the upper bounds described
above, we deduce that $\| G_{r_2} \|=0$, $\| G_{r_3} \|=1$ and $\|
G_{r_4} \|=\frac{4}{3}$.

We summarise these computations in the following proposition:

\begin{prop} \label{prop:example}
  Let $m \in \N_{\geq 2}$ and $r_m:= [\att, \btt] [\att, \btt^{-m}]$. Then
  $$
  \| G_{r_m} \| \leq  \frac{2m-4}{m-1} = 4 \cdot \Bigl( \scl_{\{ \att, \btt \}}(r) - \frac{1}{2}\Bigr).
  $$
  For $m \in \{2,3,4\}$, we have equality, i.e.,
  $\| G_{r_2} \| = 0, \| G_{r_3} \| = 1$, and $\| G_{r_4} \| = \frac{4}{3}$.
\end{prop}

\subsection{A counterexample to Question~\ref{q:main}} \label{subsec:counterexample}

Computing $\lall(r)$ is polynomial in the length of~$r$, yet it
requires lots of time, even for small words. A much more feasable
linear programming problem can be built by only considering $2$-,
$3$-, and $4$-pods. This gives an upper bound of~\lall\ and
drastically speeds up computation. Using such computations, we were
able to find a counterexample to the Main Question~\ref{q:main}:

\begin{exmp}[counterexample] \label{exmp:counterexmp}
Consider $v = \texttt{aaaabABAbaBAAbAB}$. This element
satisfies~$\scl_{\{\att,\btt\}}(v) = 5/8$. However, $\| G_v \| =
0$. To see this we will show that
$$
\cl_{\{\att,\btt\}}( v^{-1} + v + v) = 1.
$$
Indeed, we will compute that
$$
\cl_{\{\att,\btt\}} ( v^{-1} \cdot (t_1 v t_1^{-1}) \cdot (t_2 v t_2^{-1}) ) = 1 
$$
where $t_1 = \texttt{baBAAAA}$ and $t_2 =\texttt{baBaabABB}$.
In fact, we see that
$$
 v^{-1} \cdot (t_1 v t_1^{-1}) \cdot (t_2 v t_2^{-1}) 
= d \cdot [g,h] \cdot d^{-1} 
$$
where 
\begin{eqnarray*}
d &=& \texttt{baBaabABaBaa}, \\
g &=& \texttt{AAbbaBAAAAAbaBAA} \mbox{, and} \\
h &=& \texttt{bABaaaaaabABBaa}.
\end{eqnarray*}
\end{exmp}

%%%%%%%%%%%%%%%%%%%%%%%%%%%%%%%%%%%%%%%%%%%%%%%%%
%%%Bibliography:

\bibliographystyle{alpha}
\bibliography{bib_l1}

% authour info
\vfill

\noindent
\emph{Nicolaus Heuer}\\[.5em]
  {\small
  \begin{tabular}{@{\qquad}l}
   DPMMS, University of Cambridge, United Kingdom
    \\
    \textsf{nh441@cam.ac.uk},
    \textsf{https://www.dpmms.cam.ac.uk/~nh441/}
  \end{tabular}}

\medskip

\noindent
\emph{Clara L\"oh}\\[.5em]
  {\small
  \begin{tabular}{@{\qquad}l}
    Fakult\"at f\"ur Mathematik,
    Universit\"at Regensburg,
    93040 Regensburg\\
    %Germany\\
    \textsf{clara.loeh@mathematik.uni-r.de}, 
    \textsf{http://www.mathematik.uni-r.de/loeh}
  \end{tabular}}

\end{document}